\documentclass[12pt]{amsart}

\usepackage{amssymb, amscd, txfonts}
\usepackage[pdftex]{graphicx}
\usepackage[all]{xy} 

\numberwithin{equation}{section}

\newtheorem{theorem}{Theorem}[section]
\newtheorem{proposition}[theorem]{Proposition}
\newtheorem{lemma}[theorem]{Lemma}
\newtheorem{corollary}[theorem]{Corollary}
\newtheorem{step}{Step}

\theoremstyle{definition}
\newtheorem{definition}[theorem]{Definition}

\theoremstyle{remark}
\newtheorem{remark}[theorem]{Remark}
\newtheorem{claim}[theorem]{Claim}

\renewcommand{\hom}{\operatorname{Hom}}
\renewcommand{\ker}{\operatorname{Ker}}

\newcommand{\Z}{\mathbb{Z}}
\newcommand{\Q}{\mathbb{Q}}
\newcommand{\R}{\mathbb{R}}
\newcommand{\C}{\mathbb{C}}

\newcommand{\proj}{{\mathbb P}}

\newcommand{\D}{\mathcal{D}}
\newcommand{\G}{\Gamma}
\newcommand{\Gr}{\rm Gr}
\newcommand{\GF}{\Gamma_{F, \ell}}
\newcommand{\GFh}{\Gamma_{F,h}}
\newcommand{\GFhd}{\Gamma_{F,h}'}
\newcommand{\GFZ}{\Gamma(F)}
\newcommand{\UFZ}{U(F)_{\mathbb{Z}}}
\newcommand{\Xcpt}{X^{\Sigma}}
\newcommand{\DFs}{D_{F,\sigma}}
\newcommand{\DFS}{D_{F,[\sigma]}}
\newcommand{\DFscpt}{\overline{D_{F, \sigma}}}
\newcommand{\DFScpt}{\overline{D_{F, [\sigma]}}}
\newcommand{\YFcpt}{\overline{Y}_{F}}
\newcommand{\DD}{\Delta(D)}
\newcommand{\CF}{\mathcal{C}_{F}}
\newcommand{\DSG}{\Delta(\Sigma/\Gamma)}
\newcommand{\Fp}{\mathcal{F}_{p}}

\DeclareMathOperator{\aut}{Aut}
\DeclareMathOperator{\coker}{Coker}
\DeclareMathOperator{\im}{Im}

\begin{document}

\title[]{Corank spectral sequence for locally symmetric varieties}
\author[]{Shouhei Ma}
\thanks{Supported by KAKENHI 21H00971 and 20H00112} 
\address{Department~of~Mathematics, Institute~of~Science~Tokyo, Tokyo 152-8551, Japan}
\email{ma@math.titech.ac.jp}
\keywords{} 

\begin{abstract}
We construct a new type of spectral sequences for 
the mixed Hodge structures on the cohomology of locally symmetric varieties. 
These spectral sequences converge to the edge components in the Hodge triangles, 
and the $E^{1}$-terms are expressed by group cohomology associated to the cusps. 
They already degenerate at $E^1$ in a certain range, 
which gives a simple expression of some Hodge components. 
An identity of holomorphic Euler numbers is obtained as a consequence. 
\end{abstract}

\maketitle

\section{Introduction}\label{sec: intro}

Let $X={\D}/{\G}$ be a locally symmetric variety of dimension $n$, 
where ${\D}$ is a Hermitian symmetric domain and ${\G}$ is a neat arithmetic subgroup of ${\rm Aut}({\D})$. 
The singular cohomology $H^k(X)=H^k(X, {\Q})$ of $X$ is a rich object of study. 
Since $X$ is an algebraic variety, $H^k(X)$ has a canonical mixed Hodge structure. 
We denote by $(F^{\bullet}, W_{\bullet})$ the Hodge and weight filtrations. 
Since $X$ is smooth, $H^k(X)$ has weight $\geq k$. 
The pure weight $k$ part $W_kH^k(X)$ is the image from the $L^2$-cohomology (\cite{HZIII}), 
and a lot of study have been done for this part mainly in connection with discrete automorphic representations. 

On the other hand, the higher weight graded quotients ${\Gr}^{W}_{l}H^k(X)$, 
in principle related to automorphic forms in the continuous spectrum, 
should also have significance. 
A traditional approach which essentially goes back to Harder \cite{Ha2} 
is to make use of the cohomology of the Borel-Serre boundary and Eisenstein series. 
While Hodge theory of the boundary cohomology was well-established (\cite{HZII}, \cite{BW}),  
it would be fair to say that the mixed Hodge structure of $X$ itself has been studied 
still somewhat sporadically (e.g., \cite{Fr}, \cite{OS1}, \cite{OS2}, \cite{Si}). 
We wish to develop a more direct approach, 
which, when combined with knowledge from the boundary, would  
lead to a better understanding of the cohomology of $X$.  

When $k$ is small, the weight filtration on $H^k(X)$ is somewhat reduced: 
for example, $H^k(X)=W_kH^k(X)$ when $k$ is smaller than the codimension of the Baily-Borel boundary (\cite{HZIII}); 
when $k<n$, the weight filtration on $F^kH^k(X)$ is trivial (\cite{Ma2}). 
As $k$ grows, the weight filtration gets more nontrivial and interesting. 
In \cite{Ma2}, the middle degree case $k=n$ was studied in connection with Siegel operators. 

The purpose of this paper is to study the mixed Hodge structure on $H^k(X)$ in the case $k>n$. 
In what follows, let us rewrite the cohomology degree as $2n-k$ with $k\leq n$.  
The Hodge filtration $F^{\bullet}$ on $H^{2n-k}(X)$ has level $n-k \leq \bullet \leq n$. 
We will focus on the last piece $F^{n}H^{2n-k}(X)$, which is an edge of the Hodge triangle. 
It turns out that this part is already rich and at the same time relatively accessible.  
Our main result is construction of spectral sequences which compute its weight graded quotients. 

In order to state the result, let us introduce some notations. 
Let $\mathbb{G}$ be the algebraic group over ${\Q}$ with 
${\G}<\mathbb{G}({\Q})$ and $\mathbb{G}({\R})^{\circ}={\rm Aut}({\D})^{\circ}$, 
which we assume to be ${\Q}$-simple for simplicity (cf.~Remark \ref{remark: non Q-simple}). 
Let $r$ be the ${\Q}$-rank of $\mathbb{G}$. 
We have $H^{2n-k}(X)=0$ when $k<r$ (\cite{BS}). 
There is a flag $F_1\succ \cdots \succ F_r$ of reference cusps of ${\D}$ such that 
every cusp $F$ of ${\D}$ is equivalent to one of $F_i$ under the action of $\mathbb{G}({\Q})$. 
The index $i$ is called the \textit{corank} of $F$. 
We denote by $\mathcal{C}(i)$ the set of ${\G}$-equivalence classes of cusps of corank $i$. 
For a cusp $F$, let $N(F)$ be the stabilizer of $F$ in ${\rm Aut}({\D})^{\circ}$, 
$U(F)$ be the center of the unipotent radical of $N(F)$, 
and ${\GFZ}=N(F)\cap {\G}$. 
Then let ${\GF}$ be the image of ${\GFZ}\to {\aut}(U(F))$. 
On the other hand, the Siegel domain realization of ${\D}$ in the direction of $F$ 
defines a certain abelian fibration $Y_F$ over a finite cover of $F/{\GFZ}$. 
We denote by ${\YFcpt}$ an arbitrary smooth projective model of $Y_F$. 
Since ${\GF}$ acts on $Y_{F}$, $H^0(\Omega^{p}_{{\YFcpt}})$ is a ${\GF}$-module. 

We write $n(i)=\dim U(F_i)$. 
Then $0<n(1)< \cdots <n(r)\leq n$. 
We also set $n(0)=0$. 
For a ${\C}$-linear space $V$, we denote 
by $V^{\vee}$ the ${\C}$-linear dual and 
$V^{\ast}$ the $\overline{{\C}}$-linear dual. 
We can now state our main result. 

\begin{theorem}[Theorem \ref{thm: main}]\label{thm: main intro}
Let $0\leq p <n$ be fixed. 
Then there exists a first-quadrant homological spectral sequence 
\begin{equation*}\label{eqn: corank ss intro}
E^{1}_{i,m-i} = \bigoplus_{F\in \mathcal{C}(i)} H^{n(i)-m}({\GF}, H^0(\Omega^{p}_{{\YFcpt}})^{\ast}) 
\quad \Rightarrow \quad  
E^{\infty}_{m}\simeq {\Gr}^{W}_{2n-p}F^nH^{2n-p-m}(X), 
\end{equation*}
where the last isomorphism is valid for $m \geq 2$. 
We have $E^{1}_{i,j}\ne 0$ only in the range 
\begin{equation}\label{eqn: shape intro}
1\leq i \leq d(p), \qquad 0\leq j \leq n(i)-i, 
\end{equation}
where $0\leq d(p)\leq r$ is the largest corank with $n(d(p))\leq n-p$. 
\end{theorem}

We call this spectral sequence the \textit{corank spectral sequence} of level $p$. 
It computes ${\Gr}^{W}_{2n-p}F^nH^{2n-\ast}(X)$ 
from simpler group cohomology associated to the cusps. 
Note that the cohomology degree $2n-\ast$ varies while the weight $2n-p$ is fixed. 
See Figure \ref{figure: E1 range} for a shape of the range \eqref{eqn: shape intro}, 
and Figure \ref{figure: weight ss D} for a plot of the $E^{\infty}$-terms. 
A spectral sequence of the same form holds even when ${\G}$ is non-neat (\S \ref{ssec: non-neat}). 

Theorem \ref{thm: main intro} has some consequences (\S \ref{sec: complement}). 
First, by the shape \eqref{eqn: shape intro} of the $E^1$ page, 
the spectral sequence already degenerates at $E^1$ in the range $m\geq n(d(p)-1)+2$. 
This gives an isomorphism 
\begin{equation}\label{eqn: E1 degeneration intro}
{\Gr}^{W}_{2n-p}F^nH^{2n-p-m}(X) \simeq 
\bigoplus_{F\in \mathcal{C}(d(p))} H^{n(d(p))-m}({\GF}, H^0(\Omega^{p}_{{\YFcpt}})^{\ast}) 
\end{equation}
in that range (see Figure \ref{figure: E1 degeneration}). 
This generalizes a result of \cite{Ma2} in the case $k=n$ where \eqref{eqn: E1 degeneration intro} takes the form 
\begin{equation*}
{\Gr}^{W}_{n+n(i)}F^nH^{n}(X) \simeq 
\bigoplus_{F\in \mathcal{C}(i)} H^0(K_{{\YFcpt}})^{{\GF}}.  
\end{equation*}
In Theorem \ref{thm: main intro}, 
cohomology of ${\GF}$ of degree $>0$ and holomorphic forms on ${\YFcpt}$ of non-top degree arise,  
and a spectral sequence emerges instead of simple isomorphisms. 

When $p=0$, the corank spectral sequence takes the form  
\begin{equation}\label{eqn: cork ss p=o intro}
E^{1}_{i,k-i} = \bigoplus_{F\in \mathcal{C}(i)} H^{n(i)-k}({\GF}) 
\quad \Rightarrow \quad  
E^{\infty}_{k}\simeq {\Gr}^{W}_{2n}H^{2n-k}(X) 
\end{equation}
(Proposition \ref{prop: p=0}). 
This shows that the top weight pieces of $H^{2n-\ast}(X)$ 
are determined from the cohomology of the boundary groups ${\GF}$. 
For some tube domains, the next case $p=1$ is even simpler (Proposition \ref{prop: p=1}). 

Another consequence is the following Euler number identity. 

\begin{corollary}[\S \ref{ssec: Euler number}]\label{cor: Euler number intro}
We have 
\begin{equation*}
\sum_{k=r}^{n} (-1)^k \dim F^nH^{2n-k}(X) - \chi_{{\rm hol}}(\overline{X}) \: = \: 
\sum_{F} (-1)^{\dim U(F)} \chi_{{\rm hol}}({\YFcpt})\cdot \chi({\GF}). 
\end{equation*}
Here $\overline{X}$ is a smooth projective model of $X$, 
$\chi_{{\rm hol}}(V)$ is the holomorphic Euler number of a smooth projective variety $V$, and 
$\chi({\GF})$ is the Euler number of ${\GF}$. 
\end{corollary}

The alternating sum of $\dim F^nH^{2n-k}(X)$ can be viewed as an analogue of 
holomorphic Euler number for the open variety $X$. 
Corollary \ref{cor: Euler number intro} relates its difference with $\chi_{{\rm hol}}(\overline{X})$ 
to the boundary Euler numbers which are more computable.  
Indeed, in typical examples, 
$\chi({\GF})$ can be expressed by special values of Dedekind $\zeta$-functions by a theorem of Harder \cite{Ha1}. 

Let us explain an outline of the proof of Theorem \ref{thm: main intro}. 
We choose a smooth projective toroidal compactification ${\Xcpt}$ of $X$ 
whose boundary divisor $D={\Xcpt}-X$ is simple normal crossing. 
By the weight spectral sequence for $X\subset {\Xcpt}$,  
${\Gr}^{W}_{2n-p}F^nH^{\ast}(X)$ can be expressed as the homology of a certain chain complex. 
We introduce a filtration on this chain complex: 
the corank spectral sequence will be then obtained as the spectral sequence for this filtered complex. 

The best way to define and understand this filtration would be to use the dual CW complex $\Delta(D)$ of $D$. 
Our chain complex can be interpreted as 
the chain complex of a certain cellular cosheaf (in the sense of \cite{Cu}) on $\Delta(D)$. 
The point is that $\Delta(D)$ has a filtration by the corank of the underlying cusps. 
This induces the sought filtration on the chain complex. 
In this way, the corank spectral sequence can be interpreted as the spectral sequence 
that computes certain cosheaf homology of $\Delta(D)$ with the ``corank filtration'' on $\Delta(D)$. 

After this interpretation, the bulk of the proof of Theorem \ref{thm: main intro} is computation of the $E^1$-terms. 
This is a mixture of arguments in algebraic geometry, combinatorial topology and group cohomology. 
This variety of techniques would reflect the fact that 
cohomology of locally symmetric varieties sits on intersection of various branches of Mathematics. 

In the case of Hilbert modular varieties, 
the corank spectral sequence leads to a simple proof of classical results of 
Harder \cite{Ha2} and Ziegler \cite{Fr} (\S \ref{ssec: Hilbert}). 
While their approach was based on analytic continuation of Eisenstein series, this is not used here; 
weight comparison of mixed Hodge structures is used instead. 
We believe that existence of two proofs in this special case is rather a tip of an iceberg: 
mixed Hodge structures and Eisenstein series should be linked in a deeper level. 

Finally, we want to notice that the corank filtration considered in this paper has some similarity with 
the ``holomorphic rank filtration'' considered in \cite{HZIII} \S 4.4 for the cohomology of the Borel-Serre boundary. 
Although the two filtrations are for different spaces, 
one for ${\Gr}^{W}_{\bullet}H^{\ast}(X)$ and one for the boundary cohomology, 
it would be desirable to understand their relation. 

The rest of this paper is organized as follows. 
\S \ref{sec: MHS} and \S \ref{sec: modular} are recollections of 
mixed Hodge theory and toroidal compactifications respectively. 
In \S \ref{sec: corank ss}, we prove Theorem \ref{thm: main intro}. 
In \S \ref{sec: complement}, we derive some consequences and extensions. 
In \S \ref{sec: example 1}, we look at a few classical examples. 
In Appendix \ref{sec: appendix}, we prove a technical result in simplicial topology which was used in \S \ref{sec: corank ss}. 

Throughout this paper, a cohomological spectral sequence will be written as $(E_{r}^{p,q}, d_{r})$, 
while a homological spectral sequence will be written as $(E^{r}_{p,q}, d^{r})$. 
The direction of differentials is given by 
$d_{r}\colon E_{r}^{p,q} \to E_{r}^{p+r, q-r+1}$ and 
$d^{r}\colon E^{r}_{p,q} \to E^{r}_{p-r, q+r-1}$ respectively. 
We refer to \cite{Mac} Chapter XI for a comprehensive account of both types of spectral sequences.


\section{Mixed Hodge theory}\label{sec: MHS}

In this section we recall some basic facts from mixed Hodge theory which we will need later. 
A mixed Hodge structure on a ${\Q}$-linear space $V$ of finite dimension 
is a pair of an increasing filtration $W_{\bullet}$ on $V$ (the weight filtration) 
and a decreasing filtration $F^{\bullet}$ on $V_{{\C}}$ (the Hodge filtration) such that 
$F^{\bullet}$ induces a pure Hodge structure of weight $l$ on 
the weight graded quotient ${\Gr}_{l}^{W}V=W_{l}V/W_{l-1}V$ for every $l$. 
By Deligne \cite{DeII}, \cite{DeIII}, the singular cohomology $H^k(X)=H^k(X, {\Q})$ 
of a complex algebraic variety $X$ has a canonical mixed Hodge structure. 
In \S \ref{ssec: MHS smooth}, we recall the weight spectral sequence for smooth varieties. 
In \S \ref{ssec: cellular sheaf}, we explain an interpretation via combinatorial topology. 
Throughout this paper, the $r$-th Tate twist of a pure Hodge structure $V$ is denoted by $V(r)$ .

\subsection{Weight spectral sequence}\label{ssec: MHS smooth}

Let $X$ be an irreducible smooth quasi-projective variety of dimension $n$. 
We take a smooth projective compactification $X\subset \overline{X}$ such that 
the complement $D=\overline{X}-X$ is a simple normal crossing (SNC) divisor. 
We denote by $D=\sum_i D_i$ the irreducible decomposition of $D$, 
where a numbering of the irreducible components is chosen. 
For a multi-index $I=(i_1< \cdots < i_m)$ we write $D_{I}=D_{i_{1}}\cap \cdots \cap D_{i_{m}}$. 
For $1\leq j \leq m$ let 
$\rho_{j}^{I}\colon D_{I}\hookrightarrow D_{I\backslash \{ i_{j} \} }$ be the inclusion map. 
We put $D(m)=\bigsqcup_{|I|=m}D_{I}$. 

We have the weight spectral sequence 
\begin{equation*}
E_{1}^{-m,k+m} = H^{k-m}(D(m))(-m) \quad \Rightarrow \quad 
E_{\infty}^{-m,k+m}= {\Gr}^{W}_{k+m}H^k(X) 
\end{equation*}
which degenerates at $E_2$ 
(\cite{DeII}, \cite{PS} \S 4.2). 
The differential 
$d_1\colon E_{1}^{-m,k+m} \to E_{1}^{-m+1,k+m}$, 
after multiplied by $-1$, is identified with the alternating sum of the Gysin maps 
\begin{equation}\label{eqn: d1 smooth} 
\bigoplus_{|I|=m} \sum_{j=1}^{m} (-1)^{j-1} (\rho_{j}^{I})_{\ast} :  
H^{k-m}(D(m)) \to H^{k-m+2}(D(m-1)). 
\end{equation}
This shows that ${\Gr}_{F}^{p}{\Gr}^{W}_{p+q}H^k(X)\ne 0$ only in the triangle 
\begin{equation}\label{eqn: Hodge triangle}
p+q \geq k, \quad p\leq \min(k, n), \quad  q\leq \min(k, n). 
\end{equation}

Later we will use a part of this $E_1$ page. 
We write it explicitly for the convenience of reference. 

\begin{lemma}\label{lem: Fn E1}
Let $m, p\geq 0$ with $m+p\leq n$. 
Then $F^n{\Gr}^{W}_{2n-p}H^{2n-p-m}(X)$ is isomorphic to the cohomology of the complex 
\begin{equation}\label{eqn: Fn E1}
\cdots \to H^0(\Omega_{D(m+1)}^{p})^{\ast} \to H^0(\Omega_{D(m)}^{p})^{\ast} \to H^0(\Omega_{D(m-1)}^{p})^{\ast} \to \cdots 
\end{equation}
at $H^0(\Omega_{D(m)}^{p})^{\ast}$. 
Here the differentials are the complex conjugate of the dual of the signed restriction maps 
$H^0(\Omega_{D(m)}^{p}) \to H^0(\Omega_{D(m+1)}^{p})$. 
\end{lemma}

\begin{proof}
By the weight spectral sequence, 
${\Gr}^{W}_{2n-p}H^{2n-p-m}(X)=E_{\infty}^{-m, 2n-p}$ is the middle cohomology of 
\begin{equation*}
H^{2(n-m-1)-p}(D(m+1)) \to H^{2(n-m)-p}(D(m)) \to H^{2(n-m+1)-p}(D(m-1)).  
\end{equation*}
Therefore $F^{n}{\Gr}^{W}_{2n-p}H^{2n-p-m}(X)$ is the middle cohomology of 
\begin{equation*}
H^{{\rm top}, {\rm top}-p}(D(m+1)) \to H^{{\rm top}, {\rm top}-p}(D(m)) \to H^{{\rm top}, {\rm top}-p}(D(m-1)).  
\end{equation*}
It remains to notice that 
$H^{{\rm top}, {\rm top}-p}(V)$ of an irreducible smooth projective variety $V$ 
is canonically isomorphic to $H^0(\Omega_{V}^{p})^{\ast}$, 
and the Gysin map  
$H^{{\rm top}, {\rm top}-p}(V) \to H^{{\rm top}, {\rm top}-p}(W)$ 
for $V\subset W$ is the complex conjugate of the dual of the restriction map 
$H^0(\Omega_{W}^{p}) \to H^0(\Omega_{V}^{p})$. 
\end{proof}

\subsection{Interpretation via cellular cosheaves}\label{ssec: cellular sheaf}

In this subsection, we explain that the complex \eqref{eqn: Fn E1} can be interpreted in terms of 
a so-called cellular cosheaf on the dual complex of $D$. 
In essence, the appearance of cellular cosheaves in mixed Hodge theory should have been well-known.  
However, since we could not find a suitable reference in this language, let us give an account here. 

First we recall the basic theory of cellular cosheaves, 
referring the reader to \cite{Cu} Part II for more details and background. 
Classically they were known as ``coefficient systems" (see, e.g.,~\cite{Brown} p.167).  
We restrict ourselves to $\Delta$-complexes in the sense of \cite{Hatcher} \S 2.1, 
as they are the only CW complexes appearing in this paper. 

Let $\Delta$ be a finite $\Delta$-complex. 
A \textit{cellular cosheaf} on $\Delta$ is a contravariant functor $\mathcal{F}$ from 
the poset category of $\Delta$ to the category of (say) ${\C}$-linear spaces. 
This means that a ${\C}$-linear space $\mathcal{F}(\sigma)$ is assigned to each simplex $\sigma$ of $\Delta$, 
and for $\sigma\preceq \tau$ we have an ``extension map'' 
$ \mathcal{F}(\tau) \to \mathcal{F}(\sigma)$ 
which satisfies the usual compatibility conditions (dual of the axioms for presheaves). 
The chain complex $C_{\bullet}(\Delta, \mathcal{F})$ of $\mathcal{F}$ is defined by 
\begin{equation*}
C_{k}(\Delta, \mathcal{F}) = \bigoplus_{\dim \sigma =k} \mathcal{F}(\sigma), 
\end{equation*}
and the differential $C_{k}(\Delta, \mathcal{F})\to C_{k-1}(\Delta, \mathcal{F})$ 
is the alternating sum of the extension maps 
where the signs are the same as those in the ordinary simplicial homology. 
Then the homology of $\mathcal{F}$ is defined by 
\begin{equation*}
H_{k}(\Delta, \mathcal{F}) = H_{k}(C_{\bullet}(\Delta, \mathcal{F})). 
\end{equation*}

We can also consider relative homology. 
Let $\Delta'\subset \Delta$ be a sub $\Delta$-complex. 
Then $C_{\bullet}(\Delta', \mathcal{F})$ is a subcomplex of $C_{\bullet}(\Delta, \mathcal{F})$. 
The relative chain complex is defined as the quotient complex 
\begin{equation*}
C_{\bullet}(\Delta, \Delta'; \mathcal{F}) = C_{\bullet}(\Delta, \mathcal{F}) / C_{\bullet}(\Delta', \mathcal{F}). 
\end{equation*}
Naturally, we can write  
\begin{equation}\label{eqn: relative chain complex}
C_{k}(\Delta, \Delta'; \mathcal{F}) \simeq 
\bigoplus_{\substack{\dim \sigma=k \\ \sigma\not\subset \Delta'}} \mathcal{F}(\sigma). 
\end{equation}
Then the relative homology with coefficients in $\mathcal{F}$ is defined by 
\begin{equation*}
H_{k}(\Delta, \Delta'; \mathcal{F}) = H_{k}(C_{\bullet}(\Delta, \Delta'; \mathcal{F})). 
\end{equation*}

We go back to the situation $X\subset \overline{X}$, $D=\overline{X}-X$ in \S \ref{ssec: MHS smooth}. 
The dual complex of the SNC divisor $D$ is a $\Delta$-complex defined as follows (cf.~\cite{Pa} \S 2). 
To each irreducible component $D_i$ of $D$ we associate one vertex $v_i$. 
More generally, we associate an $(m-1)$-simplex $\sigma$ 
to each connected component (say $D_{\sigma}$) of a non-empty stratum $D_{I}=D_{i_1}\cap \cdots \cap D_{i_m}$ of $D(m)$. 
The vertices of $\sigma$ are $v_{i_1}, \cdots, v_{i_m}$, 
and the identification of $\sigma$ with the standard $(m-1)$-simplex (and hence an orientation of $\sigma$) 
is defined by the ordering $i_1<\cdots <i_m$. 
A simplex $\sigma$ is identified with a face of another simplex $\tau$ when $D_{\sigma}\supset D_{\tau}$. 
In this way we obtain a $\Delta$-complex, called the \textit{dual complex} of $D$ and denoted by ${\DD}$. 
Note that ${\DD}$ is not necessarily a simplicial complex because 
two different simplices may share the same set of vertices, or in other words, 
$D_{I}$ may have several connected components.  

Now we interpret the complex \eqref{eqn: Fn E1} in terms of cellular cosheaves. 
Let $p<n$. 
Let ${\Fp}$ be the cellular cosheaf on ${\DD}$ 
which associates $H^0(\Omega_{D_{\sigma}}^{p})^{\ast}$ to each simplex $\sigma$, 
and for $\sigma \preceq \tau$ the ``extension map'' 
$H^0(\Omega_{D_{\tau}}^{p})^{\ast} \to H^0(\Omega_{D_{\sigma}}^{p})^{\ast}$ 
is the conjugate dual of the restriction map 
$H^0(\Omega_{D_{\sigma}}^{p})\to H^0(\Omega_{D_{\tau}}^{p})$ 
for $D_{\tau}\subset D_{\sigma}$. 

\begin{lemma}\label{lem: cellular sheaf interpret}
The complex \eqref{eqn: Fn E1} truncated in $m\geq 1$ is identified with 
the chain complex of the cellular cosheaf ${\Fp}$ shifted by $1$. 
Hence we have an isomorphism 
\begin{equation*}
F^{n}{\Gr}^{W}_{2n-p}H^{2n-p-m}(X) \simeq H_{m-1}({\DD}, {\Fp}) 
\end{equation*}
if $m\geq 2$. 
When $m=1$, we have an exact sequence  
\begin{eqnarray}\label{eqn: m=1}
& & 0 \to F^{n}{\Gr}^{W}_{2n-p}H^{2n-p-1}(X) \to H_{0}({\DD}, {\Fp}) \\ 
& & \quad \quad \quad \to H^{0}(\Omega^{p}_{\overline{X}})^{\ast} \to F^{n}{\Gr}^{W}_{2n-p}H^{2n-p}(X) \to 0.  \nonumber
\end{eqnarray}
\end{lemma}

\begin{proof}
By the definition of ${\Fp}$, we have 
\begin{equation*}
H^0(\Omega_{D(m)}^{p})^{\ast} 
= \bigoplus_{\dim \sigma = m-1} H^0(\Omega_{D_{\sigma}}^{p})^{\ast} 
= \bigoplus_{\dim \sigma = m-1} {\Fp}(\sigma) 
= C_{m-1}({\DD}, {\Fp}). 
\end{equation*}
Both $H^0(\Omega_{D(m+1)}^{p})^{\ast} \to H^0(\Omega_{D(m)}^{p})^{\ast}$ and 
$C_{m}({\DD}, {\Fp}) \to C_{m-1}({\DD}, {\Fp})$ 
are signed sum of the Gysin maps 
$H^0(\Omega_{D_{\tau}}^{p})^{\ast} \to H^0(\Omega_{D_{\sigma}}^{p})^{\ast}$ 
for $\sigma \prec \tau$. 
The signs agree because the sign $(-1)^{j-1}$ in \eqref{eqn: d1 smooth} 
is exactly that in the simplicial homology. 
Hence the two complexes agree. 

As for the assertion in the case $m=1$, 
we rewrite the complex \eqref{eqn: Fn E1} at $m\leq 2$ as 
\begin{equation*}
C_{1}({\DD}, {\Fp}) \stackrel{\partial_{1}}{\to} C_{0}({\DD}, {\Fp}) \stackrel{\partial_{0}}{\to} 
H^{0}(\Omega^{p}_{\overline{X}})^{\ast} \to 0. 
\end{equation*}
Then the exact sequence 
\begin{equation*}
0 \to \ker(\partial_{0})/\im(\partial_{1}) \to \coker(\partial_{1}) \stackrel{\partial_{0}}{\to} 
H^{0}(\Omega^{p}_{\overline{X}})^{\ast} \to \coker(\partial_{0}) \to 0 
\end{equation*}
gives \eqref{eqn: m=1}. 
\end{proof}
 
\begin{remark}\label{remark: cosheaf Hp}
We can also consider the more basic cellular cosheaf $\mathcal{H}_{p}$ 
which assigns the pure Hodge structure of weight $-p$ 
\begin{equation*}
\mathcal{H}_{p}(\sigma) := H_{p}(D_{\sigma}) = H^{2d-p}(D_{\sigma})(d) = H^{p}(D_{\sigma})^{\vee} 
\end{equation*}
to a simplex $\sigma$ where $d=\dim D_{\sigma}$. 
The ``extension maps'' are the Gysin maps for $H_{p}$ (dual of the pullback for $H^{p}$). 
Then we have similarly 
\begin{equation*}
{\Gr}^{W}_{2n-p}H^{2n-p-m}(X)\simeq H_{m-1}({\DD}, \mathcal{H}_{p})(-n)  
\end{equation*}
for $m\geq 2$. 
Except for the cases $p=0, 1$ (\S \ref{ssec: small p}), 
this cosheaf will not be used later because the fact that $H_{p}$ is not birationally invariant when $p>1$ 
makes it hard to calculate the $E^1$ page. 
\end{remark}


\section{Toroidal compactifications}\label{sec: modular}

In this section we recall the basic theory of toroidal compactifications of locally symmetric varieties following \cite{AMRT}.

\subsection{Baily-Borel compactification}\label{ssec: BB}

Let ${\D}$ be a Hermitian symmetric domain of dimension $n$ and 
$\mathbb{G}$ be a connected semisimple linear algebraic group over ${\Q}$ such that 
$\mathbb{G}({\R})^{\circ} = {\aut}({\D})^{\circ}=:G$, 
where $\circ$ means the identity component in the classical topology. 
A rational boundary component of ${\D}$ is called a \textit{cusp} of ${\D}$. 
For two cusps $F\ne F'$, we write $F\prec F'$ if $F$ is in the closure of $F'$ in the Satake topology. 
We write $F\preceq F'$ when we want to leave the possibility $F=F'$. 

For simplicity, we assume that $\mathbb{G}$ is ${\Q}$-simple throughout this paper.  
We denote by $r$ the ${\Q}$-rank of $\mathbb{G}$. 
Then there exists a flag $F_1\succ \cdots \succ F_r$ of reference cusps of ${\D}$ such that 
every cusp $F$ is equivalent to one of $F_i$ under the action of $\mathbb{G}({\Q})$; 
more precisely, the corresponding ${\Q}$-parabolic subgroups are conjugate under $\mathbb{G}({\Q})$ 
(see \cite{BB} Theorem 3.8 and \cite{AMRT} p.141). 
We call the index $i$ the \textit{corank} of $F$, and denote it by ${\rm cork}(F)$. 
This is the only place where we use the ${\Q}$-simplicity. 
See Remark \ref{remark: non Q-simple} for the general case.

Let ${\G}$ be a neat arithmetic subgroup of $G$. 
The quotient $X={\D}/{\G}$ has the structure of a smooth quasi-projective variety, 
embedded in the Baily-Borel compactification 
$X^{bb}={\D}^{\ast}/{\G}$ as a Zariski open set (\cite{BB}). 
Here ${\D}^{\ast}$ is the union of ${\D}$ and all cusps, equipped with the Satake topology. 
Throughout this paper, we assume that the boundary of $X^{bb}$ has codimension $>1$. 
We write $X_{F}$ for the image of a cusp $F$ in $X^{bb}$ 
and use the terminology ``cusp'' also for $X_{F}$.

\subsection{Stabilizer of a cusp}\label{ssec: stabilizer}

Let $F$ be a cusp of ${\D}$. 
We denote by $N(F)<G$ the stabilizer of $F$ in $G$, 
$W(F)<N(F)$ the unipotent radical of $N(F)$, and 
$U(F)<W(F)$ the center of $W(F)$. 
Then $U(F)$ is an ${\R}$-linear space acted on by $N(F)$ by conjugation. 
We denote by $N(F)'$ the kernel of this adjoint action 
$N(F)\to {\aut}(U(F))$. 
Thus we have the filtration 
\begin{equation}\label{eqn: N(F)}
U(F) \lhd W(F) \lhd N(F)' \lhd N(F). 
\end{equation}
All these groups are defined over ${\Q}$ (\cite{AMRT} p.174). 
The quotient $V(F)=W(F)/U(F)$ is an ${\R}$-linear space. 
We have a natural map 
$N(F)'/W(F)\to {\aut}(F)$ with compact kernel and finite cokernel. 

The linear space $U(F)$ contains a distinguished open homogeneous self-adjoint cone $C(F)$. 
We denote by ${\CF}=C(F)/{\R}_{>0}^{\times}$ its projectivization. 
When $F\prec F'$, then $U(F')\subset U(F)$ and $C(F')$ is a rational boundary component of $C(F)$. 
We denote by $C(F)^{\ast}\subset U(F)$ the union of $C(F)$ and all such boundary components $C(F')$ 
(including $\{ 0 \}$ for which $F'={\D}$ by convention). 
For each $1\leq i \leq r$, we set $n(i)=\dim U(F_i)$. 
Then $0<n(1)< \cdots < n(r)\leq n$. 
We also put $n(0)=0$. 

For the given neat group ${\G}$, we denote by  
\begin{equation*}
{\UFZ} \lhd W(F)_{{\Z}} \lhd {\G}(F)' \lhd {\GFZ} 
\end{equation*}
the intersection of ${\G}$ with \eqref{eqn: N(F)}. 
In particular, ${\GFZ}$ is the stabilizer of $F$ in ${\G}$, 
and ${\UFZ}$ is a full lattice in $U(F)$. 
We will also consider the following subquotients:  
\begin{equation*}
\overline{{\GFZ}}={\GFZ}/{\UFZ}, \qquad \overline{{\GFZ}}'={\G}(F)'/{\UFZ}, 
\end{equation*}
\begin{equation*}\label{eqn: GF} 
{\GF}={\rm Im}({\GFZ}\to {\aut}(U(F))) \simeq {\GFZ}/{\G}(F)', 
\end{equation*}
\begin{equation*}
{\GFh}={\rm Im}({\GFZ}\to {\aut}(F)), \qquad {\GFhd}={\rm Im}({\G}(F)'\to {\aut}(F)). 
\end{equation*}
Then ${\GF}$ is a neat subgroup of ${\aut}(C(F))$, 
and ${\GFh}, {\GFhd}$ are neat subgroups of ${\aut}(F)$. 
Note that ${\GFhd}\simeq {\G}(F)'/W(F)_{{\Z}}$ by the neatness of ${\G}(F)'/W(F)_{{\Z}}$ (cf.~\cite{AMRT} p.174), 
so we have the exact sequence 
\begin{equation}\label{eqn: GFZd}
0 \to V(F)_{{\Z}} \to \overline{{\GFZ}}' \to {\GFhd} \to 1, 
\end{equation}
where $V(F)_{{\Z}}=W(F)_{{\Z}}/{\UFZ}$ is a full lattice in $V(F)$. 

The distinction of ${\GFh}$ and ${\GFhd}$ is somewhat subtle (and has been sometimes overlooked). 
If we denote by ${\G}_F$ the image of the natural map 
${\GFZ} \to {\aut}(U(F))\times {\aut}(F)$, 
then ${\GFhd}$ is the intersection ${\G}_F\cap {\aut}(F)$, 
while ${\GFh}$ is the image of the projection ${\G}_F\to {\aut}(F)$. 
Therefore, if we write ${\G}_{F,\ell}'={\G}_F\cap {\aut}(U(F))$, we have 
\begin{equation}\label{eqn: GFhl'}
{\GFh}/{\GFhd} \simeq {\G}_{F}/({\GFhd}\times {\G}_{F,\ell}') \simeq {\GF}/{\G}_{F,\ell}'. 
\end{equation}


\subsection{Siegel domain realization}\label{ssec: Siegel domain}

Let $F$ be a cusp of ${\D}$. 
The Siegel domain realization of ${\D}$ in the direction of $F$ is an extended $N(F)$-equivariant two-step fibration 
\begin{equation*}
{\D} \subset {\D}(F) \to \mathcal{V}_{F} \to F 
\end{equation*}
where ${\D}(F) \to \mathcal{V}_{F}$ is a principal $U(F)_{{\C}}$-bundle, 
the fibers of ${\D} \to \mathcal{V}_{F}$ are the tube domain $U(F)+iC(F)$ in $U(F)_{{\C}}$ up to base point,  
and $\mathcal{V}_{F} \to F$ is an affine space bundle. 
In the $C^{\infty}$-category, $\mathcal{V}_{F} \to F$ is a principal $V(F)$-bundle. 
We set 
\begin{equation*}
T(F)=U(F)_{{\C}}/{\UFZ}, \quad \mathcal{X}(F)={\D}/{\UFZ}, \quad \mathcal{T}(F)={\D}(F)/{\UFZ}. 
\end{equation*}
Then $T(F)$ is an algebraic torus and $\mathcal{T}(F)$ is a principal $T(F)$-bundle over $\mathcal{V}_{F}$. 
We also put 
\begin{equation*}
Y_{F} = \mathcal{V}_{F}/\overline{{\G}(F)}', \qquad X_{F}'=F/{\GFhd}. 
\end{equation*}
Then $X_{F}'$ is a finite cover of $X_F=F/{\GFh}$. 
By \eqref{eqn: GFZd}, $Y_{F}$ is a smooth abelian fibration over $X_{F}'$ 
whose fibers are $C^{\infty}$-isomorphic to $V(F)/V(F)_{{\Z}}$. 
We call $Y_{F}$ the \textit{boundary Kuga variety} over $F$. 
Since ${\GFZ}$ acts on $\mathcal{V}_{F}\to F$ equivariantly, 
the quotient ${\GF}\simeq {\GFZ}/{\G}(F)'$ acts on $Y_{F}\to X_{F}'$ equivariantly. 
The action of ${\GF}$ on $X_{F}'$ is through the finite quotient \eqref{eqn: GFhl'}.

\subsection{Toroidal compactification}\label{ssec: toroidal}

Let $\Sigma=(\Sigma_{F})_{F}$ be a ${\G}$-admissible collection of fans in the sense of \cite{AMRT} Definition III.5.1. 
Here $F$ ranges over all cusps of ${\D}$, 
and each $\Sigma_{F}$ is a ${\GF}$-admissible rational polyhedral cone decomposition of $C(F)^{\ast}$ 
satisfying suitable compatibility conditions. 
We may choose $\Sigma$ so that each $\Sigma_{F}$ is \textit{regular}, 
i.e., every cone $\sigma\in \Sigma_{F}$ is generated by a part of a ${\Z}$-basis of ${\UFZ}$. 
We call a cone $\sigma\in \Sigma_{F}$ an \textit{$F$-cone} if 
its interior is contained in $C(F)$. 
We denote by $\Sigma_{F}^{\circ}\subset \Sigma_{F}$ the set of all $F$-cones.  
If $F\prec F'$, we have $\Sigma_{F'}\subset \Sigma_{F}$ via the inclusion $C(F')^{\ast}\subset C(F)^{\ast}$. 
Then 
$\Sigma_{F}^{\circ} \subset \Sigma_{F}$ 
is the complement of $\bigcup_{F'\succ F}\Sigma_{F'}$. 

Let $T(F)^{\Sigma_{F}}$ be the torus embedding of $T(F)$ defined by the fan $\Sigma_{F}$, 
and $\mathcal{X}(F)^{\Sigma_{F}}$ be the interior of the closure of $\mathcal{X}(F)$ in 
$\mathcal{T}(F)\times_{T(F)}T(F)^{\Sigma_{F}}$. 
The toroidal compactification of $X={\D}/{\G}$ associated to $\Sigma$ is defined by the gluing 
\begin{equation*}
{\Xcpt} = \left( {\D}\sqcup \bigsqcup_{F} \mathcal{X}(F)^{\Sigma_{F}} \right) / \sim, 
\end{equation*}
where $\sim$ is the equivalence relation generated by 
the ${\G}$-actions ${\D}\to {\D}$, 
$\mathcal{X}(F)^{\Sigma_{F}}\to \mathcal{X}(\gamma F)^{\Sigma_{\gamma F}}$ 
and the projections 
${\D}\to \mathcal{X}(F)^{\Sigma_{F}}$, 
$\mathcal{X}(F')^{\Sigma_{F'}} \to \mathcal{X}(F)^{\Sigma_{F}}$ for $F\prec F'$. 

It is a fundamental theorem of Ash-Mumford-Rapoport-Tai \cite{AMRT} that 
${\Xcpt}$ is a compact Moishezon space containing $X$ as a Zariski open set, 
and the identity of $X$ extends to a morphism ${\Xcpt}\to X^{bb}$. 
When $\Sigma$ is regular, ${\Xcpt}$ is smooth and the boundary divisor $D={\Xcpt}-X$ is normal crossing. 
We may refine $\Sigma$ so that ${\Xcpt}$ is furthermore projective (\cite{AMRT} Chapter IV) 
and $D$ is simple normal crossing (see \cite{Ma2} Appendix). 
In what follows, $\Sigma$ is assumed to be chosen so.

\subsection{Structure of the boundary}\label{ssec: boundary}

We describe the structure of the SNC boundary divisor $D$ 
following \cite{AMRT} p.184 and \cite{Ma2} \S 4.1, \S 6.3. 
Let $F$ be a cusp of ${\D}$. 
It will be useful to consider the intermediate quotient 
\begin{equation*}
\mathcal{Y}(F)^{\Sigma_{F}} = \mathcal{X}(F)^{\Sigma_{F}}/\overline{{\G}(F)}', 
\end{equation*}
which fits into the diagram 
\begin{equation*}
\xymatrix@C=36pt{
\mathcal{X}(F)^{\Sigma_{F}} \ar[r]^{/\overline{{\G}(F)}'} \ar[d] & \mathcal{Y}(F)^{\Sigma_{F}}  \ar[d] \ar[r]^(0.44){/{\GF}} 
& \mathcal{X}(F)^{\Sigma_{F}}/\overline{{\GFZ}}  \ar[r] & {\Xcpt} \\ 
\mathcal{V}_{F} \ar[r]^{/\overline{{\G}(F)}'}  & Y_{F}  
}
\end{equation*}
The last map $\mathcal{X}(F)^{\Sigma_{F}}/\overline{{\GFZ}} \to {\Xcpt}$ 
is an isomorphism in a neighborhood of the $F$-stratum in ${\Xcpt}$ (\cite{AMRT} p.175). 
%
The stratification of $T(F)^{\Sigma_{F}}$ induces that of the boundary of $\mathcal{Y}(F)^{\Sigma_{F}}$ of the form  
\begin{equation*}
\mathcal{Y}(F)^{\Sigma_{F}} - \mathcal{Y}(F) = \bigsqcup_{\sigma\in \Sigma_{F}} {\DFs} 
\end{equation*}
where $\mathcal{Y}(F)=\mathcal{X}(F)/\overline{{\G}(F)}'$. 
A stratum ${\DFs}$ is mapped into the $F$-stratum in ${\Xcpt}$ if and only if $\sigma\in \Sigma_{F}^{\circ}$. 
For such $\sigma$, 
the projection $\mathcal{Y}(F)^{\Sigma_{F}}\to Y_{F}$ makes 
${\DFs}$ a principal torus bundle over $Y_{F}$ 
for the quotient torus of $T(F)$ associated to $\sigma$. 
The union $\bigsqcup_{\sigma'\succeq \sigma}D_{F,\sigma'}$ 
forms a smooth projective toric fibration over $Y_{F}$. 

The remaining group ${\GF}$ acts on $\mathcal{Y}(F)^{\Sigma_{F}}\to Y_{F}$ equivariantly 
and permutes the strata ${\DFs}$ according to its action on $\Sigma_{F}$. 
Since ${\GF}$ acts on $\Sigma_{F}^{\circ}$ freely, 
the map ${\DFs}\to {\Xcpt}$ is injective if $\sigma\in \Sigma_{F}^{\circ}$. 
We denote by ${\DFS}\subset {\Xcpt}$ its image, 
where $[\sigma]\in \Sigma_{F}^{\circ}/{\GF}$ stands for the ${\GF}$-equivalence class of $\sigma$. 
We have a natural projection 
\begin{equation}\label{eqn: DFS to YF}
{\DFS}\simeq {\DFs} \to Y_{F}, 
\end{equation}
but note that this depends on the choice of $\sigma\in \Sigma_{F}^{\circ}$ representing $[\sigma]$. 
If we use $\gamma \sigma$ with $\gamma\in {\GF}$ instead, 
this projection changes by the $\gamma$-action on $Y_{F}$. 

Now the toroidal boundary divisor $D$ is stratified as 
\begin{equation*}\label{eqn: toroidal stratification}
D =\bigsqcup_{F, [\sigma]} {\DFS}, 
\end{equation*}
where $(F, [\sigma])$ ranges over ${\G}$-equivalence classes of pairs of a cusp $F$ and an $F$-cone $\sigma$, 
or equivalently, 
pairs of a ${\G}$-equivalence class of cusps $F$ and a ${\GF}$-equivalence class of $F$-cones. 
The latter justifies our notation $(F, [\sigma])$. 

The closure ${\DFScpt}\subset {\Xcpt}$ of each stratum ${\DFS}$ is a connected component of the intersection of 
the boundary divisors corresponding to the rays of $\sigma$, 
and conversely, every connected component of $D(m)$ is of this form (\cite{Ma2} Lemma 4.3). 
This shows that 
\begin{equation}\label{eqn: D(m) toroidal}
D(m) = \bigsqcup_{\substack{F, [\sigma]\\ \dim \sigma =m}} {\DFScpt}. 
\end{equation}
The adjacent relation between the strata is given by 
\begin{equation}\label{eqn: toroidal stratify adjacent}
{\DFScpt} = {\DFS} \sqcup \bigsqcup_{F',[\sigma']}D_{F', [\sigma']}, 
\end{equation}
where $(F', [\sigma'])$ ranges over all pairs with $F'\preceq F$ and $\sigma'\succ \sigma$ 
(via the inclusion $U(F)\subseteq U(F')$).


\section{The corank spectral sequences}\label{sec: corank ss}

Let $X={\D}/{\G}$ be a locally symmetric variety of dimension $n$ as in \S \ref{sec: modular}. 
In this section we construct a series of spectral sequences that compute 
the weight graded quotients of $F^{n}H^{2n-\ast}(X)$. 
This is the central part of this paper. 

We prepare some notations. 
Let $0\leq p<n$. 
We denote by $0\leq d(p) \leq r$ the maximal corank with $n(d(p))\leq n-p$. 
By \cite{Ma2} Theorem 4.2, 
we have ${\Gr}^{W}_{2n-p}F^{n}H^{2n-p-m}(X)=0$ if $m> n(d(p))$. 
See Figure \ref{figure: weight ss D} for a shape of this condition in the case $n(r)=n$.

\begin{figure}[h]
\includegraphics[height=55mm, width=75mm]{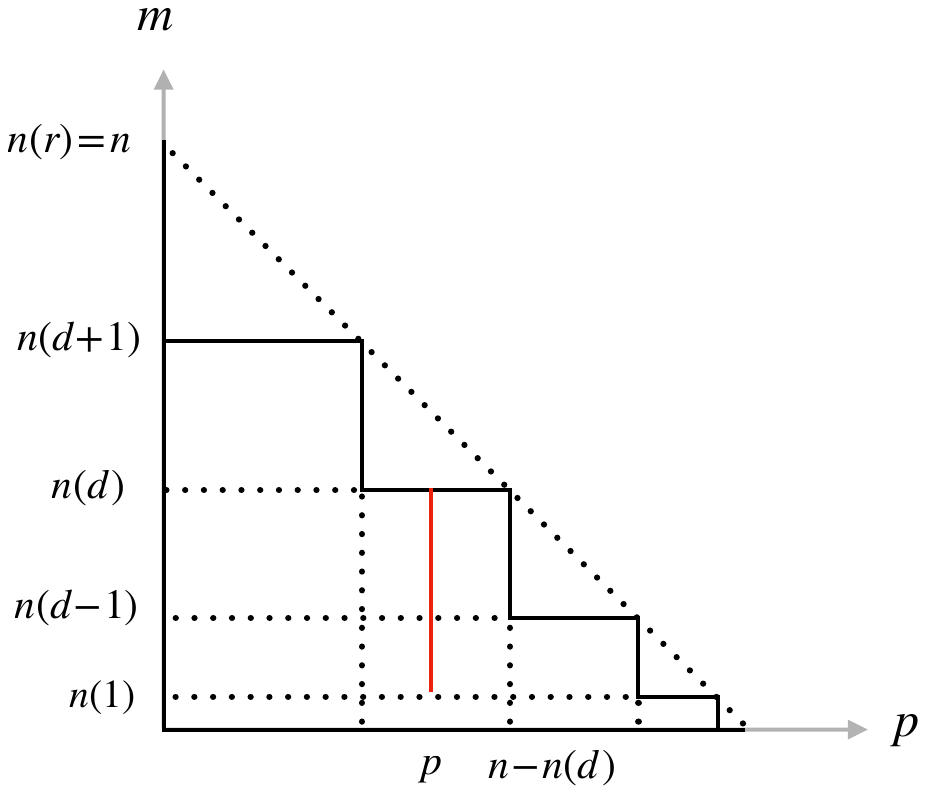}
\caption{${\Gr}^{W}_{2n-p}F^{n}H^{2n-p-m}(X)$\label{figure: weight ss D}}
\end{figure}

For each $1\leq i \leq r$, we denote by $\mathcal{C}(i)$ the set of ${\G}$-equivalence classes of cusps of corank $i$. 
For a cusp $F$, we write ${\YFcpt}$ for an arbitrary smooth projective model of the Kuga variety $Y_F$. 
Since ${\GF}$ acts on $Y_{F}$ and $H^0(\Omega_{{\YFcpt}}^{p})$ is a birational invariant, 
$H^0(\Omega_{{\YFcpt}}^{p})$ is a ${\GF}$-module. 

By a theorem of Pommerening \cite{Po}, 
we have $H^0(\Omega^{p}_{X})=H^0(\Omega^{p}_{{\Xcpt}})$ when $p<n$,  
where ${\Xcpt}$ is an arbitrary smooth projective toroidal compactification of $X$. 
Assuming that the boundary divisor $D={\Xcpt}-X$ is SNC, 
we consider the quotient space 
\begin{equation*}
H^0(\Omega^{p}_{X})_{{\rm Eis}} := H^0(\Omega^{p}_{{\Xcpt}})/ \ker( H^0(\Omega^{p}_{{\Xcpt}}) \to H^0(\Omega^{p}_{D(1)})), 
\end{equation*}
where $D(1)$ is the normalization of $D$ (as in \S \ref{ssec: MHS smooth}) and 
$H^0(\Omega^{p}_{{\Xcpt}}) \to H^0(\Omega^{p}_{D(1)})$ 
is the restriction of holomorphic forms. 
This space does not depend on the choice of ${\Xcpt}$ 
and can be regarded as the Eisenstein part of $H^0(\Omega^{p}_{X})$. 
In \cite{Ma3} Proposition 5.3, it is proved that 
a holomorphic $p$-form on $X$ vanishes at $D(1)$ in the above sense if and only if 
its restriction to the Borel-Serre boundary $\partial X^{bs}$ as a cohomology class vanishes. 
Hence we also have 
\begin{equation*}
H^0(\Omega^{p}_{X})_{{\rm Eis}} \: \simeq \: 
{\rm Im}(H^{0}(\Omega^{p}_{X}) \to H^p(\partial X^{bs})). 
\end{equation*}

We can now state our main result. 

\begin{theorem}\label{thm: main}
Let $0\leq p < n$ be fixed. 
Then there exists a first-quadrant homological spectral sequence 
\begin{equation}\label{eqn: cork ss}
E^{1}_{i,j}= \bigoplus_{F\in \mathcal{C}(i)} H^{n(i)-i-j}({\GF}, H^0(\Omega_{{\YFcpt}}^{p})^{\ast}) 
\; \; \Rightarrow \; \;  
E^{\infty}_{m} \simeq {\Gr}^{W}_{2n-p}F^{n}H^{2n-p-m}(X), 
\end{equation}
where the last isomorphism is valid for $m\geq 2$. 
For $m=1$, we have the exact sequence  
\begin{equation*}
0\to {\Gr}^{W}_{2n-p}F^{n}H^{2n-p-1}(X) \to E^{\infty}_{1} \to H^0(\Omega^{p}_{X})_{{\rm Eis}}^{\ast} \to 0. 
\end{equation*}
The support of the $E^1$ page is contained in the range 
\begin{equation}\label{eqn: cork ss E1 range}
1\leq i \leq d(p), \qquad 0\leq j \leq n(i)-i. 
\end{equation}
\end{theorem}

We call \eqref{eqn: cork ss} the \textit{corank spectral sequence} of level $p$. 
By the bound $1\leq i \leq d(p)$, this degenerates at the $E^{d(p)}$ page. 
The $E^{\infty}$-terms form the (truncated) vertical line of level $p$ in Figure \ref{figure: weight ss D}. 
See Figure \ref{figure: E1 range} for a shape of the range \eqref{eqn: cork ss E1 range}. 
Even inside \eqref{eqn: cork ss E1 range}, some terms may vanish. 
For example, the degree $1$ cohomology 
\begin{equation}\label{eqn: H1 vanish}
E^{1}_{i, n(i)-i-1} = \bigoplus _{F} H^{1}({\GF}, H^0(\Omega_{{\YFcpt}}^{p})^{\ast}) 
\end{equation}
often vanish by a theorem of Margulis (\cite{Mar} Theorem 3). 

\begin{figure}[h]
\includegraphics[height=45mm, width=60mm]{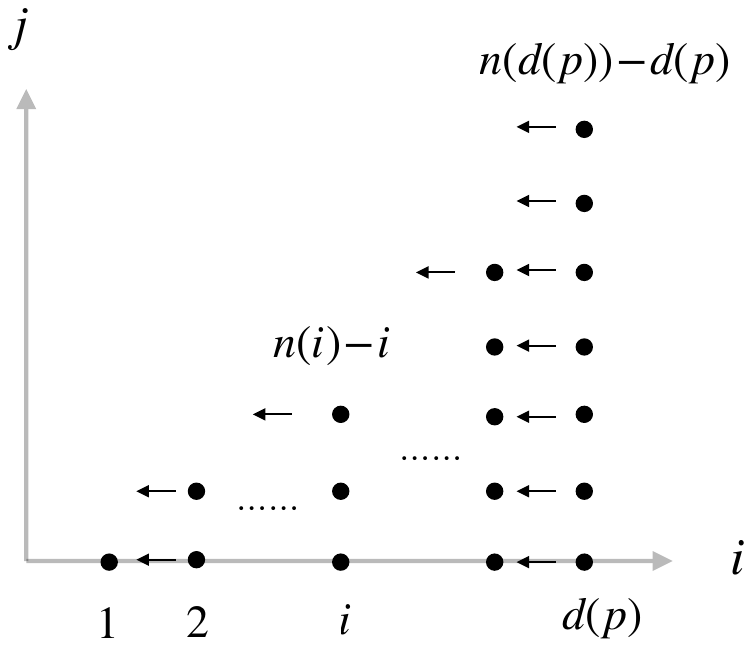}
\caption{$E^1$ page\label{figure: E1 range}}
\end{figure}

It is immediate to derive the bound \eqref{eqn: cork ss E1 range} once  
the spectral sequence \eqref{eqn: cork ss} was constructed.  
We first check $i\leq d(p)$. 
If $i> d(p)$, then $n(i)>n-p$. 
For $F\in \mathcal{C}(i)$, the Kuga variety $Y_F$ has dimension $n-n(i)<p$. 
Therefore $H^0(\Omega^{p}_{{\YFcpt}})=0$ for dimension reason. 
This shows that $E^{1}_{i,j}=0$ if $i>d(p)$. 

Next we verify $0\leq j\leq n(i)-i$. 
Let $F\in \mathcal{C}(i)$. 
Taking a flag 
$F=F_{i}\prec \cdots \prec F_{1}$ of cusps of length $i$ starting from $F$, 
we obtain the flag $U(F_{1})\subset \cdots \subset U(F_{i})=U(F)$. 
This shows that ${\GF}$ has ${\Q}$-rank $\geq i-1$.  
Since $\dim \mathcal{C}_{F}=n(i)-1$, 
${\GF}$ has cohomological dimension $\leq n(i)-i$ by the theorem of Borel-Serre \cite{BS}. 
Hence $H^{n(i)-i-j}({\GF}, V)\ne 0$ only in the range $0\leq n(i)-i-j \leq n(i)-i$ for any ${\GF}$-module $V$. 
This shows that $E^{1}_{i,j}\ne 0$ only when $0\leq j \leq n(i)-i$. 

The rest of this section is devoted to construction of the spectral sequence \eqref{eqn: cork ss}. 
In \S \ref{ssec: corank filtration}, we introduce a filtration on the dual CW complex of a toroidal boundary divisor of $X$. 
With this filtration, we obtain a spectral sequence converging to the cosheaf-valued cellular homology 
considered in Lemma \ref{lem: cellular sheaf interpret}. 
In \S \ref{ssec: E1}, we prove that the $E^{1}$-terms of this spectral sequence 
are isomorphic to those in \eqref{eqn: cork ss}, 
modulo the acyclicity of a certain simplicial pair. 
In \S \ref{ssec: acyclic}, we prove this acyclicity. 
This uses a technical construction in simplicial topology, 
whose proof will be given in Appendix \ref{sec: appendix}.

\begin{remark}\label{remark: non Q-simple}
The corank spectral sequence holds even when $\mathbb{G}$ is not ${\Q}$-simple 
if we define $n(1)< \cdots <n(r)$ as the set of possible values of $\dim U(F)$ for all cusps $F$ and 
$\mathcal{C}(i)$ as the set of ${\G}$-equivalence classes of cusps $F$ with $\dim U(F)=n(i)$ (cf.~\cite{Ma2}). 
The only difference with the ${\Q}$-simple case is the lower shape of the $E^1$ page: 
since the groups ${\GF}$ may have ${\Q}$-rank $<i-1$ even when $\dim U(F) = n(i)$, 
we may have $E^{1}_{i,j}\ne 0$ for some $j<0$.  
\end{remark}


\subsection{Corank filtration}\label{ssec: corank filtration}

We take a smooth projective toroidal compactification ${\Xcpt}$ of $X$ with SNC boundary divisor $D={\Xcpt}-X$. 
We also choose a numbering of the irreducible components of $D$. 
As explained in \S \ref{ssec: cellular sheaf}, this enables us to define the dual complex ${\DD}$ of $D$. 

To introduce a filtration on ${\DD}$, 
we first describe ${\DD}$ in terms of the fans $\Sigma=(\Sigma_{F})_{F}$. 
For $F\prec F'$, we identify $\Sigma_{F'}\subset \Sigma_{F}$ via the inclusion $U(F')\subset U(F)$. 
We glue the fans in this way. 
By our SNC condition, two rays of a cone are not ${\G}$-equivalent (\cite{Ma2} Lemma 4.3). 
Taking the quotient by ${\G}$, 
we obtain a complex of simplicial cones, which we denote by $\Sigma/{\G}$ (cf.~\cite{AMRT} p.183). 
The cones in $\Sigma/{\G}$ are ${\G}$-equivalence classes of pairs $(F, \sigma)$ 
where $F$ is a cusp and $\sigma$ is an $F$-cone. 
Replacing the cones in $\Sigma/{\G}$ by their projectivization, we obtain a $\Delta$-complex. 
We denote it by ${\DSG}$. 

\begin{lemma}
We have a natural isomorphism ${\DD}\simeq {\DSG}$. 
\end{lemma}

\begin{proof}
By the description \eqref{eqn: D(m) toroidal} of $D(m)$, 
connected components of $D(m)$ are in one-to-one correspondence with 
$m$-dimensional cones in $\Sigma/{\G}$, 
or in other words, $(m-1)$-simplices in ${\DSG}$. 
This defines a one-to-one correspondence of simplices between ${\DD}$ and ${\DSG}$. 
By \eqref{eqn: toroidal stratify adjacent}, 
this preserves the adjacent relation between the simplices. 
\end{proof}

Now we define a filtration on ${\DSG}$. 
For each $1\leq i \leq r$, let ${\DSG}_{i}$ be the union of interiors of simplices $(F, \sigma)$ with ${\rm cork}(F)\leq i$. 

\begin{lemma}
${\DSG}_{i}$ is a sub $\Delta$-complex of ${\DSG}$. 
\end{lemma}

\begin{proof}
If $\sigma'\prec \sigma$ for two cones $(F, \sigma)$, $(F', \sigma')$, 
then $F'\succeq F$ and hence ${\rm cork}(F')\leq {\rm cork}(F)$. 
This shows that if $\sigma$ belongs to ${\DSG}_{i}$, its faces also belong to ${\DSG}_{i}$. 
\end{proof}

In this way we obtain the increasing filtration 
\begin{equation}\label{eqn: corank filtration}
{\DSG}_{1} \subset {\DSG}_{2} \subset \cdots \subset {\DSG}_{r}={\DSG} 
\end{equation}
on ${\DSG}$. 
We call \eqref{eqn: corank filtration} the \textit{corank filtration}. 
In what follows, we abbreviate 
$\Delta={\DSG}={\DD}$ and $\Delta_{i}={\DSG}_{i}$. 

Let ${\Fp}$ be the cellular cosheaf on $\Delta$ introduced in \S \ref{ssec: cellular sheaf} 
which associates to a simplex $(F, \sigma)\in \Delta$ the linear space $H^{0}(\Omega_{{\DFScpt}}^{p})^{\ast}$. 
Recall from Lemma \ref{lem: cellular sheaf interpret} that 
${\Gr}^{W}_{2n-p}F^{n}H^{2n-\ast}(X)$ can be interpreted as the cellular homology of ${\Fp}$. 
Let $C_{\bullet}(\Delta, {\Fp})$ be the chain complex of ${\Fp}$. 
The corank filtration \eqref{eqn: corank filtration} defines the increasing filtration 
\begin{equation*}
C_{\bullet}(\Delta_{1}, {\Fp}) \subset C_{\bullet}(\Delta_{2}, {\Fp}) \subset \cdots 
\subset C_{\bullet}(\Delta_{r}, {\Fp}) = C_{\bullet}(\Delta, {\Fp})  
\end{equation*}
on $C_{\bullet}(\Delta, {\Fp})$. 
Its graded quotients are the relative chain complexes 
$C_{\bullet}(\Delta_{i}, \Delta_{i-1}; {\Fp})$ as defined in \S \ref{ssec: cellular sheaf}. 
By \cite{Mac} Theorem XI.3.1, this defines a homological spectral sequence 
\begin{equation}\label{eqn: filtered CW ss}
E^{1}_{i,j} = H_{i+j}(\Delta_{i}, \Delta_{i-1}; {\Fp}) 
\; \; \Rightarrow \; \;  
E^{\infty}_{i+j} = H_{i+j}(\Delta, {\Fp}).  
\end{equation}
This will be the corank spectral sequence. 
Thus the differentials $d^1$ in the corank spectral sequence are 
the connecting maps of relative homology 
for the triplets $\Delta_{i-1}\subset \Delta_{i} \subset \Delta_{i+1}$.

\subsection{The $E^1$ page}\label{ssec: E1}

The main task in the proof of Theorem \ref{thm: main} 
is computation of the relative homology in \eqref{eqn: filtered CW ss}. 
The result is as follows. 

\begin{proposition}\label{prop: E1 page}
We have a natural isomorphism 
\begin{equation*}
H_{m}(\Delta_{i}, \Delta_{i-1}; {\Fp}) \simeq 
\bigoplus_{F\in \mathcal{C}(i)} H^{n(i)-m-1}({\GF}, H^0(\Omega_{{\YFcpt}}^{p})^{\ast}). 
\end{equation*}
\end{proposition}

Theorem \ref{thm: main} is immediate from Proposition \ref{prop: E1 page} as follows. 
By Proposition \ref{prop: E1 page}, 
the spectral sequence \eqref{eqn: filtered CW ss} takes the form  
\begin{equation*}\label{eqn: filtered CW spectral sequence}
E^{1}_{i,j} = 
\bigoplus_{F\in \mathcal{C}(i)} H^{n(i)-i-j-1}({\GF}, H^0(\Omega_{{\YFcpt}}^{p})^{\ast})
\; \; \Rightarrow \; \;  
E^{\infty}_{m} = H_{m}(\Delta, {\Fp}).  
\end{equation*}
By Lemma \ref{lem: cellular sheaf interpret}, we have 
$E^{\infty}_{m} \simeq {\Gr}^{W}_{2n-p}F^{n}H^{2n-p-m-1}(X)$ 
if $m>0$. 
For $m=0$ we have the exact sequence 
\begin{equation*}
0\to {\Gr}^{W}_{2n-p}F^{n}H^{2n-p-1}(X) \to E_{0}^{\infty} \to 
\im(H^0(\Omega^{p}_{D(1)})^{\ast}\to H^0(\Omega^{p}_{{\Xcpt}})^{\ast}) \to 0 
\end{equation*}
by \eqref{eqn: m=1}.  
The signed Gysin map 
$H^0(\Omega^{p}_{D(1)})^{\ast}\to H^0(\Omega^{p}_{{\Xcpt}})^{\ast}$ 
has the same image as the non-signed Gysin map. 
Taking the conjugate dual, we obtain 
\begin{equation*}
\im(H^0(\Omega^{p}_{D(1)})^{\ast}\to H^0(\Omega^{p}_{{\Xcpt}})^{\ast}) 
\; \simeq \; H^0(\Omega^{p}_{X})_{{\rm Eis}}^{\ast}. 
\end{equation*}
Finally, shifting the index $j$ by $1$, we obtain the spectral sequence \eqref{eqn: cork ss}. 

The rest of \S \ref{ssec: E1} is devoted to the proof of Proposition \ref{prop: E1 page}. 
Let us prepare some notations. 
For a cusp $F$, we denote by $\Delta_F$ the infinite simplicial complex 
obtained as the projectivization of the cone complex $\Sigma_{F}$. 
The numbering of the irreducible components of $D$ induces that of vertices of every simplex of $\Delta_{F}$. 
We will use the symbol $\sigma$ for both simplices in $\Delta_{F}$ and cones in $\Sigma_{F}$. 
This will not cause confusion except for dimension;  
whenever we write $\dim \sigma$ in what follows, 
we shall always mean the dimension as a simplex. 
We write $\Delta_{F}^{\circ}\subset \Delta_{F}$ for the subset of simplices corresponding to $F$-cones. 
We call such simplices \textit{$F$-simplices}. 
We denote by $\partial \Delta_F\subset \Delta_{F}$ the complement of $\Delta_{F}^{\circ}$. 
This is the projectivization of the boundary cone complex $\bigcup_{F'\succ F}\Sigma_{F'}$.  
The simplicial complex $\Delta_{F}$ is locally finite outside $\partial \Delta_{F}$, but not so at $\partial \Delta_{F}$. 

The group ${\GF}$ acts on $\Delta_{F}$. 
Since two vertices of a simplex in $\Delta_{F}$ are not ${\G}$-equivalent, 
they are in particular not ${\GF}$-equivalent. 
This enables us to form the quotient 
$\Delta_{F}'=\Delta_{F}/{\GF}$ as a finite $\Delta$-complex. 
We denote by $\partial \Delta_{F}'=\partial \Delta_{F}/{\GF}$ its boundary. 
We use the symbol $[\sigma]$ for expressing simplices in $\Delta_{F}'$, where $\sigma\in \Delta_{F}$. 
The numbering of vertices of $[\sigma]$ is induced from that of $\sigma$. 
This is well-defined because the ${\GF}$-action on $\Delta_{F}$ preserves the numbering of vertices of simplices by construction. 

Now we proceed to the proof of Proposition \ref{prop: E1 page}. 

\begin{proof}[(Proof of Proposition \ref{prop: E1 page})]
We divide the proof into three steps. 
The first step separates the cusps, 
the second step resolves ${\GF}$-monodromy, 
and the last step is appearance of group cohomology.  

\begin{step}\label{step1}
We have 
\begin{equation*}\label{eqn: split relative cohomology}
H_{m}(\Delta_{i}, \Delta_{i-1}; {\Fp}) = 
\bigoplus_{F\in \mathcal{C}(i)}  H_{m}(\Delta_{F}', \partial\Delta_{F}'; {\Fp}). 
\end{equation*}
\end{step}

\begin{proof}
By \eqref{eqn: relative chain complex}, 
the simplices $(F, [\sigma])$ that contribute to the chain complex $C_{\bullet}(\Delta_{i}, \Delta_{i-1}; {\Fp})$ 
are those with ${\rm cork}(F)=i$. 
For each such cusp $F$, we have a natural simplicial map $\Delta_{F}'\to \Delta_{i}$. 
This maps $\partial\Delta_{F}'$ to a subcomplex of $\Delta_{i-1}$, 
and sends simplices in $\Delta_{F}' \backslash \partial\Delta_{F}'$ bijectively to $F$-simplices in $\Delta_{i}\backslash \Delta_{i-1}$ 
because ${\G}$-equivalence reduces to ${\GF}$-equivalence for $F$-simplices. 
If $F\ne F'\in \mathcal{C}(i)$, no $F'$-simplex appears in the boundary of a $F$-simplex and vice versa. 
This shows that the map 
$\bigsqcup_{F\in \mathcal{C}(i)}(\Delta_{F}', \partial\Delta_{F}') \to (\Delta_{i}, \Delta_{i-1})$ 
induces a splitting    
\begin{equation*}
C_{\bullet}(\Delta_{i}, \Delta_{i-1}; {\Fp}) = 
\bigoplus_{F\in \mathcal{C}(i)}  C_{\bullet}(\Delta_{F}', \partial\Delta_{F}'; {\Fp}) 
\end{equation*}
of chain complex. 
\end{proof}

The problem is thus reduced to calculating 
$H_{\ast}(\Delta_{F}', \partial\Delta_{F}'; {\Fp})$ 
for each cusp $F$. 

\begin{step}\label{step2}
$H_{m}(\Delta_{F}', \partial\Delta_{F}'; {\Fp})$ 
is isomorphic to the homology of the complex 
\begin{equation}\label{eqn: complex step2}
\hom (C^{\bullet}_{c}(\Delta_{F}, \partial\Delta_{F}), \: H^0(\Omega_{{\YFcpt}}^{p})^{\ast})^{{\GF}} 
\end{equation}
in degree $m$, 
where $C^{\bullet}_{c}(\Delta_{F}, \partial\Delta_{F})$ is the compactly supported cochain complex 
of the simplicial pair $(\Delta_{F}, \partial\Delta_{F})$. 
\end{step}

\begin{proof}
By \eqref{eqn: relative chain complex} and the definition of ${\Fp}$, we can write 
\begin{equation}\label{eqn: CmDFFp}
C_{m}(\Delta_{F}', \partial\Delta_{F}'; {\Fp}) = 
\bigoplus_{\substack{ [\sigma]\in \Delta_{F}' \backslash \partial\Delta_{F}' \\ \dim \sigma=m}}  H^{0}(\Omega^{p}_{{\DFScpt}})^{\ast}. 
\end{equation}
The boundary map of $C_{\bullet}(\Delta_{F}', \partial\Delta_{F}'; {\Fp})$ 
is alternating sum of conjugate dual of the restriction maps 
$H^{0}(\Omega^{p}_{\overline{D_{F,[\tau]}}}) \to H^{0}(\Omega^{p}_{{\DFScpt}})$ 
for $\tau \prec \sigma$. 
Here the sign agrees with that determined by the numbering of vertices of $[\sigma]$ 
because $\Delta_{F}'\to \Delta_{i}\subset \Delta$ preserves such numberings by construction. 

We have a monodromy action of ${\GF}$ on the strata ${\DFS}$, 
which prevents us from separating the holomorphic part and the combinatorial part in \eqref{eqn: CmDFFp}. 
In order to resolve this, we go back to the boundary of $\mathcal{Y}(F)^{\Sigma_{F}}$ considered in \S \ref{ssec: boundary}. 
Thus, for each $[\sigma]\in \Delta_{F}' \backslash \partial\Delta_{F}'$, 
we identify ${\DFS}$ with the quotient of $\bigsqcup_{\sigma}D_{F,\sigma}$ by ${\GF}$ 
where $\sigma$ ranges over all simplices in $\Delta_{F}^{\circ}$ representing $[\sigma]$. 
This identifies $H^{0}(\Omega^{p}_{{\DFScpt}})$ with 
the ${\GF}$-invariant part of $\prod_{\sigma} H^{0}(\Omega^{p}_{{\DFscpt}})$,  
where ${\DFscpt}$ is a smooth projective compactification of $D_{F,\sigma}$.  
Then \eqref{eqn: CmDFFp} can be rewritten as 
\begin{equation*}
C_{m}(\Delta_{F}', \partial\Delta_{F}'; {\Fp}) \simeq  
\left( \prod_{\substack{ \sigma \in \Delta_{F}^{\circ} \\ \dim \sigma=m}}  H^{0}(\Omega^{p}_{{\DFscpt}})^{\ast} \right)^{{\GF}}. 
\end{equation*}

Now, as explained in \S \ref{ssec: boundary}, 
each ${\DFs}$ has a \textit{canonical} torus fibration ${\DFs}\to Y_{F}$. 
Recall that $H^0(\Omega^{p})$ is a stable birational invariant of a smooth projective variety, i.e., 
a birational map $X_1 \dashrightarrow X_{2}\times {\proj}^1$ induces an isomorphism 
$H^0(\Omega^{p}_{X_{2}}) \simeq H^0(\Omega^{p}_{X_{1}})$.  
Therefore pullback by ${\DFs}\to Y_{F}$ gives a canonical isomorphism 
$H^{0}(\Omega^{p}_{{\YFcpt}}) \to H^{0}(\Omega^{p}_{{\DFscpt}})$. 
We introduce a formal symbol $\langle \sigma \rangle$ expressing the simplex $\sigma$ and rewrite this isomorphism as 
\begin{equation}\label{eqn: normalize DFs}
H^{0}(\Omega^{p}_{{\DFscpt}}) = H^{0}(\Omega^{p}_{{\YFcpt}}) \otimes {\C}\langle \sigma \rangle. 
\end{equation}
Passing to the conjugate dual, we write 
\begin{equation*}
H^{0}(\Omega^{p}_{{\DFscpt}})^{\ast} = \hom ({\C}\langle \sigma \rangle^{\vee}, H^{0}(\Omega^{p}_{{\YFcpt}})^{\ast}) 
\end{equation*}
similarly. 
Then we have 
\begin{eqnarray*}
\prod_{\sigma}  H^{0}(\Omega^{p}_{{\DFscpt}})^{\ast} 
& = & 
\prod_{\sigma}\hom ({\C}\langle \sigma \rangle^{\vee}, H^{0}(\Omega^{p}_{{\YFcpt}})^{\ast}) \\ 
& \simeq & 
\hom \left( \bigoplus_{\sigma} {\C}\langle \sigma \rangle^{\vee}, \: H^{0}(\Omega^{p}_{{\YFcpt}})^{\ast} \right) \\ 
& = &  
\hom ( C^{m}_{c}(\Delta_{F}, \partial\Delta_{F}), \: H^{0}(\Omega^{p}_{{\YFcpt}})^{\ast} ), 
\end{eqnarray*}
where $\sigma$ ranges over $m$-simplices in $\Delta_{F}^{\circ}$. 
The passage from direct product to direct sum will be crucial later. 
This identification is ${\GF}$-equivariant, where 
${\GF}$ acts on both factors 
$H^{0}(\Omega^{p}_{{\YFcpt}})^{\ast}$ and $C^{m}_{c}(\Delta_{F}, \partial\Delta_{F})$. 
Hence we obtain an isomorphism 
\begin{equation}\label{eqn: chain complex resolv}
C_{\bullet}(\Delta_{F}', \partial\Delta_{F}'; {\Fp}) \simeq  
\hom ( C^{\bullet}_{c}(\Delta_{F}, \partial\Delta_{F}), \: H^{0}(\Omega^{p}_{{\YFcpt}})^{\ast} )^{{\GF}} 
\end{equation}
of linear spaces. 

It remains to verify that \eqref{eqn: chain complex resolv} is an isomorphism of chain complexes. 
If $\tau\in \Delta_{F}^{\circ}$ is a codimension $1$ face of $\sigma\in \Delta_{F}^{\circ}$, 
then ${\DFs}\to Y_{F}$ is a sub torus fibration of the smooth toric fibration 
$\bigsqcup_{\tau'\succeq \tau}D_{F,\tau'}\to Y_{F}$. 
Therefore the restriction map 
$H^{0}(\Omega^{p}_{\overline{D_{F,\tau}}}) \to H^{0}(\Omega^{p}_{{\DFscpt}})$ 
is identified via the normalization \eqref{eqn: normalize DFs} with the map 
\begin{equation*}
{\rm id}\otimes (\langle \tau \rangle \mapsto \langle \sigma \rangle) :  
H^{0}(\Omega^{p}_{{\YFcpt}}) \otimes {\C}\langle \tau \rangle 
\to H^{0}(\Omega^{p}_{{\YFcpt}}) \otimes {\C}\langle \sigma \rangle. 
\end{equation*}
Since the numbering of vertices of simplices of $\Delta_{F}'$ is induced from that for $\Delta_{F}$, 
the signs in the boundary map of $C_{\bullet}(\Delta_{F}', \partial\Delta_{F}'; {\Fp})$ 
agree with those of $C^{\bullet}_{c}(\Delta_{F}, \partial\Delta_{F})$.   
\end{proof}

\begin{step}\label{step3}
The degree $m$ homology of \eqref{eqn: complex step2} is isomorphic to 
the group cohomology 
$H^{n(i)-1-m}({\GF}, H^{0}(\Omega^{p}_{{\YFcpt}})^{\ast})$. 
\end{step}
 
\begin{proof}
The key point is calculation of the cohomology of $C^{\bullet}_{c}(\Delta_{F}, \partial\Delta_{F})$: 
\begin{equation}\label{eqn: acyclic 1}
H^{k}_{c}(\Delta_{F}, \partial\Delta_{F}) \simeq 
\begin{cases}
{\C} & k=n(i)-1, \\ 
0 & k\ne n(i)-1. 
\end{cases}
\end{equation}
We will prove this in the next \S \ref{ssec: acyclic} (Proposition \ref{prop: acyclic}). 
Admitting this, the proof of Step \ref{step3} proceeds as follows. 
Since $C^{m}_{c}(\Delta_{F}, \partial\Delta_{F})$ is the direct \textit{sum} over $m$-simplices in $\Delta_{F}^{\circ}$ and 
${\GF}$ acts on $\Delta_{F}^{\circ}$ feely, 
we see that $C^{m}_{c}(\Delta_{F}, \partial\Delta_{F})$ is a free ${\GF}$-module. 
Then \eqref{eqn: acyclic 1} means that 
the $(n(i)-1)$-shift of $C^{\bullet}_{c}(\Delta_{F}, \partial\Delta_{F})$, 
viewed as a \textit{chain} complex in nonpositive degree, is a free resolution of ${\C}$. 
Therefore, by the definition of group cohomology, 
the degree $m$ homology of \eqref{eqn: complex step2}  
is the degree $n(i)-1-m$ cohomology of ${\GF}$ with coefficients in $H^{0}(\Omega^{p}_{{\YFcpt}})^{\ast}$. 
\end{proof}

The proof of Proposition \ref{prop: E1 page} is now completed 
modulo the proof of \eqref{eqn: acyclic 1}, 
which is postponed to the next \S \ref{ssec: acyclic}. 
\end{proof}

\subsection{Acyclicity}\label{ssec: acyclic}

We prove the following property 
which was used in the proof of Proposition \ref{prop: E1 page}.

\begin{proposition}\label{prop: acyclic}
Let $F$ be a cusp of corank $i$. 
Then we have 
\begin{equation}\label{eqn: acyclic 2}
H^{k}_{c}(\Delta_{F}, \partial\Delta_{F}) \simeq  
\begin{cases}
{\C} & k=n(i)-1, \\ 
0 & k\ne n(i)-1. 
\end{cases}
\end{equation}
\end{proposition}

A naive idea to deduce this is as follows. 
We identify simplicial cohomology with singular cohomology. 
\textit{Suppose} that $\Delta_{F}$ was locally finite also at $\partial\Delta_{F}$. 
Then the topological space underlying $\Delta_{F}$ is locally compact, 
so we have 
\begin{equation*}
H^{k}_{c}(\Delta_{F}, \partial\Delta_{F}) \simeq H^{k}_{c}(\Delta_{F}-\partial\Delta_{F}) 
\end{equation*}
by \cite{Bredon} Proposition II.12.3, 
where $\Delta_{F}-\partial\Delta_{F}$ means the complement of 
the topological space underlying $\partial\Delta_{F}$ in that underlying $\Delta_{F}$. 
Since $\Delta_{F}-\partial\Delta_{F}$ is homeomorphic to $\mathcal{C}_{F}\simeq {\R}^{n(i)-1}$, 
we arrive at the compactly supported cohomology of ${\R}^{n(i)-1}$. 

However, $\Delta_{F}$ is in fact not locally finite at $\partial\Delta_{F}$, 
so this naive argument does not work. 
In order to modify this, we replace $(\Delta_{F}, \partial\Delta_{F})$. 
In Proposition \ref{prop: retraction}, 
we will construct a pair $(\Delta_{F}^{+}, \partial\Delta_{F}^{+})$ of simplicial complexes such that 
(1) $\Delta_{F}^{+}$ is locally finite, 
(2) $\Delta_{F}^{+}-\partial\Delta_{F}^{+}$ is homeomorphic to $\Delta_{F}-\partial\Delta_{F}$, and 
(3) $H^{k}_{c}(\Delta_{F}^{+}, \partial\Delta_{F}^{+})$ is isomorphic to $H^{k}_{c}(\Delta_{F}, \partial\Delta_{F})$ for every $k$. 
Roughly speaking, $(\Delta_{F}^{+}, \partial\Delta_{F}^{+})$ is a retraction of $(\Delta_{F}, \partial\Delta_{F})$ toward interior. 
The actual construction is technical and makes use of the barycentric subdivision of $\Delta_{F}$. 
Since this purely belongs to simplicial topology, we decided to treat it separately in Appendix \ref{sec: appendix}. 

Admitting Proposition \ref{prop: retraction}, it is now easy to prove Proposition \ref{prop: acyclic}. 

\begin{proof}[(Proof of Proposition \ref{prop: acyclic})]
Since $H^{k}_{c}(\Delta_{F}, \partial\Delta_{F})\simeq H^{k}_{c}(\Delta_{F}^{+}, \partial\Delta_{F}^{+})$, 
it suffices to calculate $H^{k}_{c}(\Delta_{F}^{+}, \partial\Delta_{F}^{+})$. 
Since $\Delta_{F}^{+}$ is locally finite, we are now able to use \cite{Bredon} Proposition II.12.3 to deduce 
\begin{equation*}
H^{k}_{c}(\Delta_{F}^{+}, \partial\Delta_{F}^{+}) \simeq H^{k}_{c}(\Delta_{F}^{+}-\partial\Delta_{F}^{+}). 
\end{equation*}
Furthermore, we have the homeomorphisms 
\begin{equation*}
\Delta_{F}^{+}-\partial\Delta_{F}^{+} \simeq \Delta_{F}-\partial\Delta_{F} 
\simeq \mathcal{C}_{F} \simeq {\R}^{n(i)-1}. 
\end{equation*}
Therefore we obtain 
\begin{equation*}
H^{k}_{c}(\Delta_{F}^{+}, \partial\Delta_{F}^{+}) \simeq  H^{k}_{c}({\R}^{n(i)-1}) \simeq 
\begin{cases}
{\C} & k=n(i)-1, \\ 
0 & k\ne n(i)-1. 
\end{cases}
\end{equation*}
This proves Proposition \ref{prop: acyclic}. 
\end{proof}

The proof of Theorem \ref{thm: main} is thus completed modulo Proposition \ref{prop: retraction}, 
which will be proved in Appendix \ref{sec: appendix}.

\section{Complements}\label{sec: complement}

In this section we derive some consequences and extensions of the corank spectral sequences. 
In \S \ref{ssec: degeneration}, we look at some range of early degeneration. 
In \S \ref{ssec: small p}, we look at the case $p\leq 1$. 
In \S \ref{ssec: Euler number}, we derive an Euler number identity. 
In \S \ref{ssec: non-neat}, we give an extension to the non-neat case. 
These four subsections can be read independently.

\subsection{Degeneration}\label{ssec: degeneration}

We fix $0\leq p \leq n-n(1)$. 
Let $(E^{r}_{p,q}, d^{r})$ be the corank spectral sequence \eqref{eqn: cork ss} of level $p$.  
As we can see from Figure \ref{figure: E1 range}, when $m$ is close to $n(d(p))$, 
there are only few columns in the $E^1$ page that contribute to $E^{\infty}_{m}$, 
so we have early degeneration. 
We observe this more explicitly in \S \ref{sssec: 5.1.1} and \S \ref{sssec: 5.1.2}. 
As $m$ decreases, more columns contribute, so we need more pages to arrive at $E^{\infty}$. 
However, when $n(1)=1$, we again have early degeneration for small $m$. 
We observe this in \S \ref{sssec: 5.1.4}. 
In what follows, we abbreviate $d=d(p)$.

\subsubsection{$E^1$ degeneration}\label{sssec: 5.1.1}

We have $E^1$ degeneration in the following range. 

\begin{proposition}\label{prop: E1 degeneration}
When $n(d-1)+2 \leq m \leq n(d)$, we have 
\begin{equation*}
F^n{\Gr}^{W}_{2n-p}H^{2n-p-m}(X) \simeq 
\bigoplus_{F\in \mathcal{C}(d)} H^{n(d)-m}({\GF}, H^0(\Omega_{{\YFcpt}}^{p})^{\ast}). 
\end{equation*}
\end{proposition}

\begin{proof}
By the range \eqref{eqn: cork ss E1 range} of the $E^1$ page, 
when $n(d-1)+2\leq m \leq n(d)$, there is no nonzero differential $d^r$ that hits to nor starts from $E^{r}_{d,m-d}$, 
and there is no other $E^{1}_{i,j}\ne 0$ contributing to $E^{\infty}_{m}$. 
Therefore $E^{\infty}_{m} = E^{1}_{d,m-d}$. 
\end{proof}

See Figure \ref{figure: E1 degeneration} for a shape of the range of this $E^1$ degeneration 
on the $(p, m)$-plane plotting $F^n{\Gr}^{W}_{2n-p}H^{2n-p-m}(X)$. 

\begin{figure}[h]
\includegraphics[height=55mm, width=70mm]{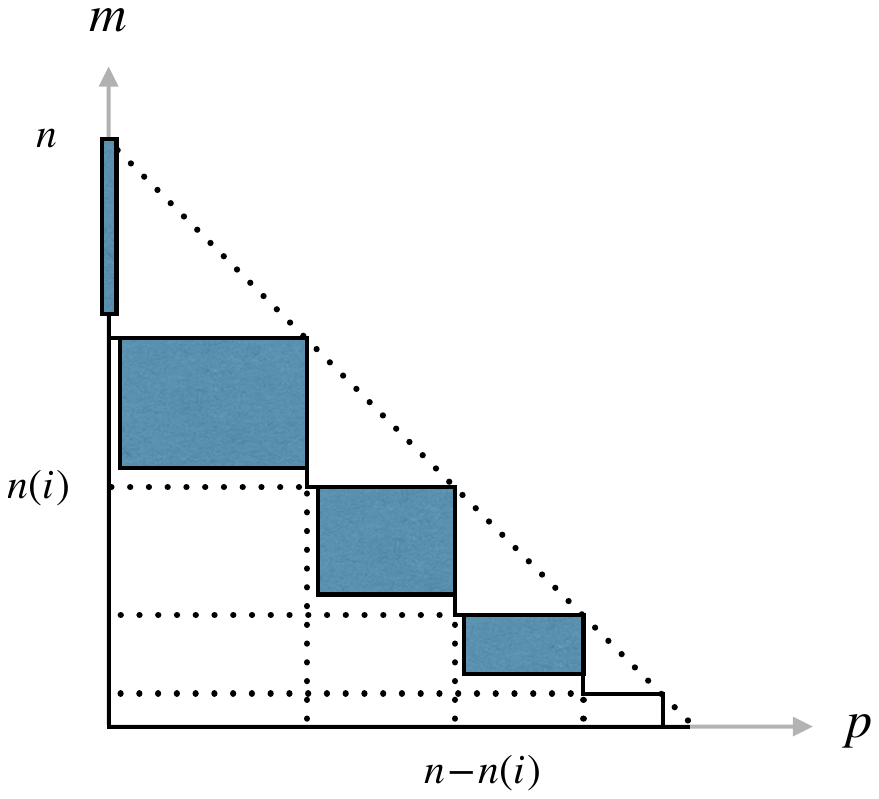}
\caption{Range of $E^1$ degeneration\label{figure: E1 degeneration}}
\end{figure}

\subsubsection{$E^{3/2}$ degeneration}\label{sssec: 5.1.2}

In the next range $n(d-2)+2 \leq m\leq n(d-1)+1$, 
the next column $i=d-1$ comes to contribute, 
so the spectral sequence degenerates only at $E^2$. 
Nevertheless, the situation around the starting point $m=n(d-1)+1$ is still close to $E^1$ degeneration. 
Indeed, there is no nonzero $d^{r}$ that hits to $E^{r}_{d, n(d-1)-d+1}$, and 
$d^{1}$ is the only nonzero differential that starts from $E^{r}_{d, n(d-1)-d+1}$. 
It follows that  
\begin{equation}\label{eqn: E3/2 degeneration}
E^{\infty}_{n(d-1)+1} = \ker (  E^{1}_{d,n(d-1)-d+1} \stackrel{d^1}{\to} E^{1}_{d-1,n(d-1)-d+1}). 
\end{equation}
Explicitly, this is written as 
\begin{eqnarray}\label{eqn: E3/2}
& & F^n{\Gr}^{W}_{2n-p}H^{2n-p-n(d-1)-1}(X) \simeq \\ 
& &  \ker \left(
\bigoplus_{F\in \mathcal{C}(d)} H^{n(d)-n(d-1)-1}({\GF}, H^0(\Omega_{{\YFcpt}}^{p})^{\ast}) \to 
\bigoplus_{F\in \mathcal{C}(d-1)} (H^0(\Omega_{{\YFcpt}}^{p})^{\ast})^{{\GF}} \right). \nonumber 
\end{eqnarray}
Furthermore, when $E^{1}_{d-1, n(d-1)-d}$ vanishes (cf.~\eqref{eqn: H1 vanish}), 
we can extend \eqref{eqn: E3/2 degeneration} to the exact sequence 
\begin{equation*}
0 \to E^{\infty}_{n(d-1)+1} \to E^{1}_{d,n(d-1)-d+1} \to E^{1}_{d-1,n(d-1)-d+1} \to 
E^{\infty}_{n(d-1)} \to E^{1}_{d, n(d-1)-d} \to 0 
\end{equation*}
if $n(d-2)+2 \leq n(d-1)$. 


\subsubsection{The case $m=2$}\label{sssec: 5.1.4}

In some typical examples we have $n(1)=1$. 
In this case, as can be seen from Figure \ref{figure: E1 range}, we have early degeneration also around $m=2$. 
Indeed, there is no nonzero differential $d^{r}$ with $r\geq 2$ that hits to nor starts from $E^{r}_{2,0}$, 
and we have no other $E^{1}_{i,j}\ne 0$ that contributes to $E^{\infty}_{2}$. 
Consequently, we have 
\begin{equation*}
E^{2}_{2,0} \simeq E^{\infty}_{2} \simeq F^{n}{\Gr}^{W}_{2n-p}H^{2n-p-2}(X). 
\end{equation*}
This $E^{2}_{2,0}$ is the middle homology of 
$E^{1}_{1,0} \leftarrow E^{1}_{2,0} \leftarrow E^{1}_{3,0}$, 
which is explicitly written as 
\begin{equation*}
\bigoplus_{F\in \mathcal{C}(1)} (H^0(\Omega_{{\YFcpt}}^{p})^{\ast})^{{\GF}} \leftarrow 
\bigoplus_{F\in \mathcal{C}(2)} H^{n(2)-2}({\GF}, H^0(\Omega_{{\YFcpt}}^{p})^{\ast}) \leftarrow 
\bigoplus_{F\in \mathcal{C}(3)} H^{n(3)-3}({\GF}, H^0(\Omega_{{\YFcpt}}^{p})^{\ast}). 
\end{equation*}

\subsection{The case $p\leq 1$}\label{ssec: small p}

In this subsection we look at the corank spectral sequences of level $p=0, 1$. 
Recall that what played a key role in Step \ref{step2} of the proof of Proposition \ref{prop: E1 page} 
was the stable birational invariance of $H^{0}(\Omega^{p})$. 
On the other hand, when $p\leq 1$, the singular cohomology $H^p(-, {\Q})$ is also a stable birational invariant. 
This enables us to derive a refinement of the corank spectral sequences of level $p\leq 1$. 

We begin with the case $p=0$. 

\begin{proposition}\label{prop: p=0}
The corank spectral sequence of level $p=0$ takes the form 
\begin{equation*}
E^{1}_{i,k-i} = \bigoplus_{F\in \mathcal{C}(i)} H^{n(i)-k}({\GF}) 
\quad \Rightarrow \quad  
E^{\infty}_{k}\simeq {\Gr}^{W}_{2n}H^{2n-k}(X), 
\end{equation*}
where the last isomorphism is valid when $k\geq 2$. 
This is defined for cohomology with ${\Q}$-coefficients. 
\end{proposition}

\begin{proof}
We have $H^0(\Omega_{{\YFcpt}}^{0})={\C}$ and 
$F^{n}{\Gr}^{W}_{2n}H^{2n-k}(X)={\Gr}^{W}_{2n}H^{2n-k}(X)$. 
So the spectral sequence \eqref{eqn: cork ss} for $p=0$ takes the above form. 
To prove that it is defined over ${\Q}$, 
we can run the proof of Theorem \ref{thm: main} with the cosheaf 
$\mathcal{H}_{0}$ (Remark \ref{remark: cosheaf Hp}) instead of $\mathcal{F}_{0}$. 
The stable birational invariance makes it possible to apply the argument of Step \ref{step2} to $\mathcal{H}_{0}$. 
Other parts remain valid. 
This produces the desired spectral sequence over ${\Q}$.  
\end{proof}

Similarly, in the case $p=1$, we have the following. 

\begin{proposition}\label{prop: p=1 Hodge} 
We have a spectral sequence 
\begin{equation}\label{eqn: H1 corank ss}
E^{1}_{i,m-i}= \bigoplus_{F\in \mathcal{C}(i)}  H^{n(i)-m}({\GF}, H^1({\YFcpt})^{\vee}(-n)) \\ 
\quad \Rightarrow \quad 
E^{\infty}_{m} \simeq {\Gr}^{W}_{2n-1}H^{2n-m-1}(X) 
\end{equation}
of pure Hodge structures of weight $2n-1$, 
where the last isomorphism is valid for $m\geq 2$. 
\end{proposition}

\begin{proof}
This can be obtained by reworking the proof of Theorem \ref{thm: main} with the cosheaf 
$\mathcal{H}_{1}$ instead of $\mathcal{F}_{1}$. 
\end{proof}

This refines the corank spectral sequence of level $p=1$  
which can be obtained by taking the $F^{n}$-part of \eqref{eqn: H1 corank ss}. 

When ${\D}$ has a tube domain model, i.e., $n(r)=n$, 
we can sometimes simplify \eqref{eqn: H1 corank ss} as follows. 
Let $F$ be a cusp of ${\D}$. 
The group $N(F)/W(F)$ acts on the real vector space $V(F)=W(F)/U(F)$ by conjugation. 
In particular, the group ${\GFhd}\simeq {\G}(F)'/W(F)_{{\Z}}$ acts on $V(F)$. 
We consider the following condition: 
\begin{equation}\label{eqn: condition}
V(F)^{\vee} \; \textrm{has no nonzero} \; {\GFhd}\textrm{-invariant part}. 
\end{equation}
When the identity component $(N(F)'/W(F))^{\circ}$ of $N(F)'/W(F)$ has no compact factor, 
it is semisimple and isogenous to $\aut(F)$ (see \cite{AMRT} p.174). 
Then, by the Borel density theorem, \eqref{eqn: condition} holds if 
$V(F)$ has no nonzero $(N(F)'/W(F))^{\circ}$-invariant part (see \cite{Ra} Chapter V).  
This holds, e.g., for Siegel modular varieties. 

\begin{proposition}\label{prop: p=1}
Assume that $n(r)=n$, 
the condition \eqref{eqn: condition} is satisfied for every cusp of corank $<r$, 
and $\aut(F)$ is ${\Q}$-simple for every cusp of corank $<r-1$. 
Then we have an isomorphism
\begin{equation}\label{eqn: Gr2n-1}
{\Gr}^{W}_{2n-1}H^{2n-k}(X) \simeq 
\bigoplus_{F\in \mathcal{C}(r-1)} (H^{n(r-1)-k+1}({\G}_{F,\ell}') \otimes H^1(\overline{X_{F}'})^{\vee})^{G}(-n) 
\end{equation}
of pure Hodge structures of weight $2n-1$. 
Here $\overline{X_{F}'}$ is a smooth projective model of $X_{F}'=F/{\GFhd}$ and  
$G={\GF}/{\G}_{F,\ell}'\simeq {\GFh}/{\GFhd}$. 
\end{proposition}

\begin{proof}
We look at the $E^1$ page of the spectral sequence \eqref{eqn: H1 corank ss}. 
Let $F$ be a cusp of corank $<r$. 
By the Leray spectral sequence for the projection $\pi\colon Y_{F}\to X_{F}'$, we have the exact sequence 
\begin{equation*}
0 \to H^1(X_{F}') \stackrel{\pi^{\ast}}{\to} H^1(Y_{F}) \to H^0(X_{F}', R^1\pi_{\ast}{\Q}). 
\end{equation*}
Since in the $C^{\infty}$-level $\pi$ is a principal fiber bundle with fiber $V(F)/V(F)_{{\Z}}$, 
the local system $R^1\pi_{\ast}{\Q}$ is the one associated to the representation $V(F)_{{\Q}}^{\vee}$ of ${\GFhd}$. 
Then the condition \eqref{eqn: condition} implies 
\begin{equation*}
H^0(X_{F}', R^1\pi_{\ast}{\Q}) = (V(F)_{{\Q}}^{\vee})^{{\GFhd}} = 0. 
\end{equation*}
Hence $H^1(Y_F) \simeq H^1(X_{F}')$. 
Taking the ${\Gr}^{W}_{1}$-part gives $H^1({\YFcpt})\simeq H^1(\overline{X_{F}'})$. 

When $F$ has corank $r$, $Y_{F}$ is a point by the assumption $n(r)=n$ and so $H^1(Y_{F})=0$. 
When $F$ has corank $<r-1$, $\aut(F)$ has ${\Q}$-rank $\geq2$. 
With our ${\Q}$-simplicity assumption,   
we have $H^1(X_{F}')=0$ by a theorem of Margulis \cite{Mar}. 
Hence $H^1(Y_{F})=0$ also in this case. 
These show that the $E^{1}$ page of \eqref{eqn: H1 corank ss} is supported on the line  $i=r-1$. 
Hence $E^{\infty}_{k-1}\simeq E^{1}_{r-1,k-r}$. 
Explicitly, this is written as 
\begin{equation*}
{\Gr}^{W}_{2n-1}H^{2n-k}(X) \simeq 
\bigoplus_{F\in \mathcal{C}(r-1)} H^{n(r-1)-k+1}({\GF}, H^1(\overline{X_{F}'})^{\vee}(-n)).  
\end{equation*}

We write $V=H^1(\overline{X_{F}'})^{\vee}$ and $l=n(r-1)-k+1$. 
Since $G$ is finite, the Hochschild-Serre spectral sequence for 
$0\to {\G}_{F,\ell}'\to {\GF} \to G \to 0$ with coefficients in a ${\Q}$-linear space degenerates at $E_{2}$. 
Hence we have 
\begin{equation}\label{eqn: HS inv}
H^{l}({\GF}, V) \simeq H^0(G, H^{l}({\G}_{F,\ell}', V)) = H^{l}({\G}_{F,\ell}', V)^{G} 
=  (H^{l}({\G}_{F,\ell}')\otimes V)^{G}. 
\end{equation}
Here the last equality holds because ${\G}_{F,\ell}'$ acts on $V$ trivially. 
\end{proof}


\subsection{Holomorphic Euler number}\label{ssec: Euler number}

In this subsection we prove an Euler number identity. 
We first derive a more reduced version than Corollary \ref{cor: Euler number intro}. 
We write 
\begin{equation*}
F^nH^{2n-k}(X)_{{\rm Eis}} := F^nH^{2n-k}(X)/W_{2n-k}F^nH^{2n-k}(X) 
\end{equation*}
for the non-pure part of $F^nH^{2n-k}(X)$. 
The holomorphic Euler number of a smooth projective variety $V$ is denoted by  
$\chi_{{\rm hol}}(V)=\sum_{k}(-1)^{k}h^{k,0}(V)$. 
The Euler number of a group $G$ is denoted by $\chi(G)=\sum_{k}(-1)^{k}h^{k}(G)$. 

\begin{proposition}\label{prop: Euler number}
We have 
\begin{eqnarray}\label{eqn: Euler number}
& & \sum_{k=r}^{n}(-1)^k \dim F^nH^{2n-k}(X)_{{\rm Eis}} 
\, - \, \sum_{k=0}^{n-1}(-1)^k \dim H^{0}(\Omega_{X}^{k})_{{\rm Eis}} \\ 
& = & \sum_{F} (-1)^{\dim U(F)} \chi_{{\rm hol}}({\YFcpt}) \cdot \chi({\GF}), \nonumber
\end{eqnarray}
where $F$ ranges over all cusps of $X^{bb}$.  
\end{proposition}

\begin{proof}
By the general theory of spectral sequences, 
the Euler number  
$\chi(E^r) = \sum_{i,j}(-1)^{i+j}\dim E^{r}_{i,j}$ 
of a spectral sequence $(E^{r}_{i,j})$ does not change with respect to $r$. 
We apply this to the corank spectral sequences. 

First we fix $0\leq p \leq n-1$. 
Let $({}_pE^{r}_{i,j})$ be the spectral sequence \eqref{eqn: cork ss} of level $p$. 
We calculate both sides of $\chi({}_pE^1)=\chi({}_pE^{\infty})$. 
As for ${}_pE^{\infty}$, we have 
\begin{equation}\label{eqn: pEinfinity}
\chi({}_pE^{\infty}) \; = \; 
\sum_{m\geq 1} (-1)^m \dim F^{n}{\Gr}^{W}_{2n-p}H^{2n-p-m}(X) - \dim H^{0}(\Omega^{p}_{X})_{{\rm Eis}} 
\end{equation}
by the description of ${}_pE^{\infty}_{m}$ in Theorem \ref{thm: main}. 
As for ${}_pE^{1}$, we have  
\begin{eqnarray*}
\chi({}_pE^{1}) 
& = & \sum_{i,j} \sum_{F\in \mathcal{C}(i)} (-1)^{i+j} \dim H^{n(i)-i-j}({\GF}, H^0(\Omega^{p}_{{\YFcpt}})^{\ast}) \\ 
& = & \sum_{F} (-1)^{\dim U(F)} \sum_{k} (-1)^{k} \dim H^{k}({\GF}, H^0(\Omega^{p}_{{\YFcpt}})^{\ast}).  
\end{eqnarray*}
In the second line we wrote $k=n(i)-i-j$ and $F$ ranges over all cusps of $X^{bb}$. 
Since ${\GF}$ is torsion-free, we have 
\begin{equation*}
\sum_{k} (-1)^{k} \dim H^{k}({\GF}, H^0(\Omega^{p}_{{\YFcpt}})^{\ast}) = 
\chi({\GF}) \cdot h^{p,0}({\YFcpt}) 
\end{equation*}
by \cite{Brown} Corollary 1. 
Therefore 
\begin{equation}\label{eqn: chipE1}
\chi({}_pE^{1}) = \sum_{F} (-1)^{\dim U(F)} \chi({\GF}) \cdot h^{p,0}({\YFcpt}). 
\end{equation} 
 
Next we take the alternating sum of $\chi({}_pE^1)=\chi({}_pE^{\infty})$ over $0\leq p \leq n-1$. 
As for ${}_pE^{1}$, we have  
\begin{eqnarray*}
\sum_{p=0}^{n-1} (-1)^{p} \chi({}_pE^{1}) 
& = & \sum_{F} (-1)^{\dim U(F)} \sum_{p=0}^{n-1} (-1)^{p} h^{p,0}({\YFcpt}) \cdot \chi({\GF}) \\  
& = & \sum_{F} (-1)^{\dim U(F)} \chi_{{\rm hol}}({\YFcpt}) \cdot \chi({\GF}). 
\end{eqnarray*}
This is the right hand side of \eqref{eqn: Euler number}. 
On the other hand, 
$\sum_{p} (-1)^{p} \chi({}_pE^{\infty})$ is equal to   
\begin{equation*}
\sum_{p=0}^{n-1} \sum_{m\geq 1} (-1)^{p+m} \dim F^{n}{\Gr}^{W}_{2n-p}H^{2n-p-m}(X) 
- \sum_{p=0}^{n-1} (-1)^{p} \dim H^0(\Omega^{p}_{X})_{{\rm Eis}}. 
\end{equation*}
If we write $k=p+m$, the first term can be written as 
\begin{eqnarray*}
\sum_{p=0}^{n-1} \sum_{k>p} (-1)^{k} \dim F^{n}{\Gr}^{W}_{2n-p}H^{2n-k}(X)  
& = & \sum_{k\geq 1} (-1)^{k} \sum_{p=0}^{k-1} \dim F^{n}{\Gr}^{W}_{2n-p}H^{2n-k}(X) \\ 
& = & \sum_{k=1}^{n} (-1)^{k} \dim F^{n}H^{2n-k}(X)_{{\rm Eis}}. 
\end{eqnarray*}
Therefore $\sum_{p}(-1)^p \chi({}_pE^{\infty})$ is equal to the left hand side of \eqref{eqn: Euler number}. 
\end{proof}

The intermediate equality $\chi({}_pE^1)=\chi({}_pE^{\infty})$ in the proof 
would also deserve noticing especially when $p=0$. 

\begin{corollary}
We have 
\begin{equation*}\label{eqn: Euler number p=0}
-1 + \sum_{k=r}^{n} (-1)^{k} \dim {\Gr}^{W}_{2n}H^{2n-k}(X) 
= \sum_{F} (-1)^{\dim U(F)} \chi({\GF}). 
\end{equation*}
\end{corollary}

\begin{proof}
The left hand side is $\chi({}_0E^{\infty})$ by \eqref{eqn: pEinfinity}, 
and the right hand side is $\chi({}_0E^1)$ by \eqref{eqn: chipE1}. 
\end{proof}

Corollary \ref{cor: Euler number intro} is equivalent to Proposition \ref{prop: Euler number}. 
Indeed, by the weight spectral sequence, we have 
\begin{eqnarray*}
& & \dim F^nH^{2n-k}({\Xcpt}) - \dim F^nW_{2n-k}H^{2n-k}(X) \\ 
& = & \dim \im (F^{n-1}H^{2n-2-k}(D(1)) \stackrel{\textrm{Gysin}}{\longrightarrow} F^nH^{2n-k}({\Xcpt})) \\ 
& = & \dim \im (H^{0}(\Omega^{k}_{{\Xcpt}}) \to H^{0}(\Omega^{k}_{D(1)})). 
\end{eqnarray*}
The last expression is equal to 
$\dim H^0(\Omega^{k}_{X})_{{\rm Eis}}$ when $k<n$, 
and vanishes when $k=n$. 
Therefore the left hand side of \eqref{eqn: Euler number} can be rewritten as 
\begin{equation*}
\sum_{k=r}^{n}(-1)^k \dim F^nH^{2n-k}(X) - \chi_{{\rm hol}}({\Xcpt}). 
\end{equation*}
This is the left hand side of the identity in Corollary \ref{cor: Euler number intro}. 

\begin{remark}
One may also consider the alternating sum of $\dim F^kH^k(X)$,  
but since $F^kH^k(X)=F^kH^k({\Xcpt})$ when $k<n$ (\cite{Ma2}), 
this is almost identical to $\chi_{{\rm hol}}({\Xcpt})$. 
\end{remark}

\subsection{Non-neat case}\label{ssec: non-neat}

In this subsection we consider the case when ${\G}< {\rm Aut}({\D})$ is no longer neat. 
The quotient $X={\D}/{\G}$ is a quasi-projective variety with at most quotient singularities. 
By \cite{DeIII}, its singular cohomology still has a canonical mixed Hodge structure $(F^{\bullet}, W_{\bullet})$. 

The definitions of the arithmetic groups in \S \ref{ssec: stabilizer} still make sense. 
The quotient $Y_{F}=\mathcal{V}_{F}/\overline{{\G}(F)}'$ has singularities in general, but the group ${\GF}$ still acts on it. 
We denote by ${\YFcpt}$ a smooth projective model of $Y_{F}$. 
Then $H^{0}(\Omega^{p}_{{\YFcpt}})$ is a ${\GF}$-module. 

\begin{proposition}\label{prop: non-neat}
A spectral sequence of the same form as \eqref{eqn: cork ss} holds 
even when ${\G}$ is non-neat. 
\end{proposition}

For the proof of Proposition \ref{prop: non-neat}, 
we need the following lemma in group cohomology. 

\begin{lemma}\label{lem: group cohomology}
Let $0\to G_1 \to G_2 \to G_3 \to 0$ be an exact sequence of groups 
and $H_2$ be a normal subgroup of $G_2$ of finite index. 
We put $H_1=H_2\cap G_1$ and $H_3=H_2/H_1$. 
Let $V$ be a $G_2/H_1$-module which is a ${\Q}$-linear space. 
Then we have a natural isomorphism 
\begin{equation}\label{eqn: HSHS}
H^{k}(H_3, V)^{G_2/H_2} \: \simeq \: H^{k}(G_3, V^{G_1/H_1}). 
\end{equation}
\end{lemma}

\begin{proof}
Since $G_2/H_2$ is finite, its cohomology in degree $>0$ with coefficients in a ${\Q}$-linear space vanish. 
Therefore the Hochschild-Serre spectral sequence for 
$0\to H_3\to G_2/H_1 \to G_2/H_2 \to 0$ degenerates and gives 
\begin{equation*}
H^k(H_3, V)^{G_2/H_2} = H^{0}(G_2/H_2, H^k(H_3, V)) \simeq H^k(G_2/H_1, V). 
\end{equation*}
Similarly, by the Hochschild-Serre for 
$0\to G_1/H_1\to G_2/H_1 \to G_3 \to 0$ 
with $G_1/H_1$ finite, we obtain 
\begin{equation*}
H^k(G_2/H_1, V) \simeq H^k(G_3, H^{0}(G_1/H_1, V)) = H^{k}(G_3, V^{G_1/H_1}). 
\end{equation*}
This proves Lemma \ref{lem: group cohomology}. 
\end{proof}

Now we prove Proposition \ref{prop: non-neat}. 

\begin{proof}[(Proof of Proposition \ref{prop: non-neat})]
We choose a neat normal subgroup ${\G}_{0}\lhd {\G}$ of finite index. 
The quotient group $G={\G}/{\G}_{0}$ acts on the smooth variety $X_{0}={\D}/{\G}_{0}$ with $X=X_{0}/G$. 
The mixed Hodge structure on $H^k(X)$ is identified with the $G$-invariant part 
of the mixed Hodge structure on $H^k(X_{0})$ (see \cite{Ma2} \S 8). 
We use the notations 
${\G}_{0}(F), {\G}_{0}(F)', {\G}_{F,\ell}^{0}, Y_{F}^{0}, \mathcal{C}(i)_{0}$ 
for ${\G}_{0}$ with obvious meaning. 

By taking the toroidal compactification of $X_{0}$ in a $G$-invariant way, 
we obtain a $G$-equivariant corank spectral sequence for ${\G}_{0}$. 
Taking its $G$-invariant part, we obtain a spectral sequence of the form 
\begin{equation*}\label{eqn: cork ss non-neat}
E^{1}_{i,m-i} = 
\left( \bigoplus_{F\in \mathcal{C}(i)_{0}} H^{n(i)-m}({\G}_{F,\ell}^{0}, H^0(\Omega_{\overline{Y}_{F}^{0}}^{p})^{\ast}) \right)^{G}
\; \; \Rightarrow \; \;  
E^{\infty}_{m} \simeq {\Gr}^{W}_{2n-p}F^{n}H^{2n-p-m}(X). 
\end{equation*}

We shall rewrite the $E^{1}$-terms. 
We abbreviate $V_{F}=H^0(\Omega_{\overline{Y}_{F}^{0}}^{p})^{\ast}$. 
The group $G$ acts on $\mathcal{C}(i)_{0}$ with quotient $\mathcal{C}(i)$. 
For each $F\in \mathcal{C}(i)_{0}$, we denote by $G_{F}<G$ its stabilizer. 
Then we have 
\begin{equation*}
\left( \bigoplus_{F\in \mathcal{C}(i)_{0}} H^{k}({\G}_{F,\ell}^{0}, V_{F}) \right)^{G} 
\: \: \simeq \: \:  
\bigoplus_{F\in \mathcal{C}(i)} H^{k}({\G}_{F,\ell}^{0}, V_{F})^{G_{F}}. 
\end{equation*}
Now we apply Lemma \ref{lem: group cohomology} to 
$G_{1}={\G}(F)'$, $G_{2}={\GFZ}$, $G_{3}={\GF}$ and $H_2={\G}_{0}(F)$. 
Then $H_{1}={\G}_{0}(F)'$, $H_{3}={\G}_{F,\ell}^{0}$ and $G_{2}/H_{2}\simeq G_{F}$. 
Therefore \eqref{eqn: HSHS} takes the form  
\begin{equation*}
H^{k}({\G}_{F,\ell}^{0}, V_{F})^{G_{F}} \simeq H^{k}({\GF}, V_{F}^{{\GFZ}'/{\G}_{0}(F)'}). 
\end{equation*}
Since ${\GFZ}'/{\G}_{0}(F)'$ acts on $Y_{F}^{0}$ with quotient $Y_{F}$, we have 
$V_{F}^{{\GFZ}'/{\G}_{0}(F)'}=H^0(\Omega_{{\YFcpt}}^{p})^{\ast}$. 
\end{proof}


\section{Examples}\label{sec: example 1}

In this section we look at the corank spectral sequences 
for some classical modular varieties. 

\subsection{Hilbert modular varieties}\label{ssec: Hilbert}

Let $X={\D}/{\G}$ be a Hilbert modular variety of dimension $n>1$, 
where ${\D}$ is the $n$-fold product of the upper half plane and 
${\G}$ is a neat arithmetic subgroup of ${\rm SL}(2, K)$ for a totally real number field $K$ of degree $n$. 
The Baily-Borel compactification $X^{bb}$ has only $0$-dimensional cusps with $n(1)=n$. 
Hence $H^{k}(X)$ is pure when $k<n$ (\cite{HZII}). 

Harder \cite{Ha2} studied the restriction map 
$r_{\ast} \colon H^{\ast}(X) \to H^{\ast}(\partial X^{bs})$ 
to the Borel-Serre boundary $\partial X^{bs}$. 
By using analytic continuation of Eisenstein series, 
he proved that $r_{2n-k}$ is surjective when $2\leq k \leq n$. 
Based on Harder's result, Ziegler (\cite{Fr} \S III.7) proved that 
${\Gr}^{W}_{l}H^{2n-k}(X)=0$ for $2n-k<l<2n$ and 
${\Gr}^{W}_{2n}H^{2n-k}(X)$ is mapped isomorphically to $H^{2n-k}(\partial X^{bs})$ 
when $2\leq k \leq n$. 
In what follows, as an application of the corank spectral sequence, 
we give a purely Hodge-theoretic proof of the results of Harder and Ziegler which does not use Eisenstein series. 

Let $2\leq k \leq n$. 
The corank spectral sequence of level $p=0$ in Proposition \ref{prop: p=0} 
degenerates at $E^{1}$ and hence 
\begin{equation*}\label{eqn: dim Hilbert}
\dim {\Gr}^{W}_{2n}H^{2n-k}(X) = \sum_{F} \dim H^{n-k}({\GF}). 
\end{equation*}
On the other hand, it is known that 
\begin{equation*}\label{eqn: BS cohomology Hilbert}
\dim H^{2n-k}( \partial_{F}X^{bs} ) = \dim H^{n-k}({\GF}), 
\end{equation*}
where $\partial_{F}X^{bs}$ is the component of $\partial X^{bs}$ over $X_{F}$ 
(see \cite{Ha2} p.142 or \cite{Fr} Proposition III.2.1). 
It follows that 
\begin{equation}\label{eqn: dim equality Hilbert}
\dim {\Gr}^{W}_{2n}H^{2n-k}(X) = \dim H^{2n-k}(\partial X^{bs}). 
\end{equation}

Next we look at the fundamental exact sequence (\cite{Ha2}) 
\begin{equation*}
H^{2n-k}_{c}(X) \stackrel{i_{2n-k}}{\to} H^{2n-k}(X) \stackrel{r_{2n-k}}{\to} H^{2n-k}(\partial X^{bs}). 
\end{equation*}
Since $i_{2n-k}$ is a morphism of mixed Hodge structures and 
$H^{2n-k}_{c}(X)$ has weight $\leq 2n-k$ (see \cite{PS} \S 5.5.2), 
we have ${\rm Im}(i_{2n-k}) \subset W_{2n-k}H^{2n-k}(X)$. 
Therefore 
\begin{eqnarray*}
\dim {\Gr}^{W}_{2n}H^{2n-k}(X) & \leq & \dim (H^{2n-k}(X)/{\rm Im}(i_{2n-k})) \\ 
& = & \dim {\rm Im}(r_{2n-k}) \; \leq \; \dim H^{2n-k}(\partial X^{bs}). 
\end{eqnarray*}
Then \eqref{eqn: dim equality Hilbert} shows that the two inequalities here are actually equalities. 
In this way, we find simultaneously that 
$r_{2n-k}$ is surjective, 
${\Gr}^{W}_{l}H^{2n-k}(X)=0$ for $2n-k<l<2n$, 
and ${\rm Im}(i_{2n-k}) = W_{2n-k}H^{2n-k}(X)$. 

This proof can be regarded as a simplification of the argument of Ziegler. 
Indeed, \eqref{eqn: dim Hilbert} was also proved in \cite{Fr} p.193 -- p.198 by a more direct argument. 
The novelty here is to reverse the order of argument 
from surjectivity $\Rightarrow$ MHS to MHS $\Rightarrow$ surjectivity. 

Although $H^{\ast}(\partial X^{bs})$ has a natural mixed Hodge structure (\cite{HZII}), 
notice that we did not make use of it in our argument.

\subsection{Siegel modular varieties}\label{ssec: Siegel}

Let $X={\D}/{\G}$ be a Siegel modular variety of degree $g>1$, 
where ${\D}$ is the Siegel upper half space of degree $g$ and 
${\G}$ is a neat arithmetic subgroup of ${\rm Sp}(2g, {\Q})$. 
We have $n=g(g+1)/2$ and the ${\Q}$-rank is $g$. 
A cusp $F$ of corank $g'\leq g$ corresponds to a rational isotropic subspace $I$ of ${\Q}^{2g}$ of dimension $g'$. 
Then $X_{F}'$ is a Siegel modular variety of degree $g-g'$ associated to $I^{\perp}/I$, 
$Y_{F}$ is isogenous to the $g'$-fold self-product of a universal family of abelian varieties over $X_{F}'$, 
$U(F)_{{\Q}}$ is identified with ${\rm Sym}^2I$, 
$C(F)$ is the cone of positive-definite quadratic forms, and 
${\GF}$ is a neat subgroup of ${\rm SL}(I)\simeq {\rm SL}(g', {\Q})$. 
In particular, we have $n(g')=g'(g'+1)/2$. 
 
When $p=\dim Y_{F}$, $H^0(\Omega^{p}_{{\YFcpt}})=H^0(K_{{\YFcpt}})$ is isomorphic to 
the space $S\!_{g+1}({\GFhd})$ of Siegel cusp forms of weight $g+1$ on $X_{F}'$ (see \cite{Hatada}, \cite{Ma1}). 
The ${\GF}$-action on $H^0(K_{{\YFcpt}})$ comes from the action of 
$G={\GFh}/{\GFhd}\simeq {\GF}/{\G}_{F,\ell}'$ on $S\!_{g+1}({\GFhd})$. 
Hence we have 
\begin{equation*}
H^k({\GF}, H^0(K_{{\YFcpt}})^{\ast}) \simeq 
(H^k({\G}_{F,\ell}') \otimes S\!_{g+1}({\GFhd})^{\ast})^{G} 
\end{equation*}
by the same argument as \eqref{eqn: HS inv}. 
Similarly, when $p=1$, 
the proof of Proposition \ref{prop: p=1} shows that 
$H^0(\Omega^{1}_{{\YFcpt}}) \ne 0$ only when $g'=g-1$, 
and in that case, 
$H^0(\Omega^{1}_{{\YFcpt}})\simeq H^0(\Omega^{1}_{\overline{X_{F}'}})$ is isomorphic to 
the space of cusp forms of weight $2$ on the modular curve $X_{F}'$. 
Then we have similarly 
\begin{equation*}
H^k({\GF}, H^0(\Omega^{1}_{{\YFcpt}})^{\ast}) \simeq 
(H^k({\G}_{F,\ell}') \otimes S\!_{2}({\GFhd})^{\ast})^{G}. 
\end{equation*}
%
%
For general $1<p<\dim Y_{F}$, holomorphic $p$-forms on ${\YFcpt}$ are in principle controlled by 
vector-valued Siegel modular forms on $X_{F}'$ via the holomorphic Leray spectral sequence for $Y_{F}\to X_{F}'$. 
The $E^{1}$-terms of the corank spectral sequences are more complicated mixture of 
group cohomology and modular forms. 

Finally, by the formula of Harder \cite{Ha1}, we have 
\begin{equation*}
\chi({\GF}) = C\cdot \prod_{k=2}^{g'} \zeta(1-k) = 0 
\end{equation*}
when $g'\geq 3$ where $C$ is a suitable constant. 
It follows that only cusps of corank $g' = 1, 2$ survive in the right hand side of \eqref{eqn: Euler number}.

\subsection{Orthogonal modular varieties}\label{ssec: orthogonal}

Let $V$ be a rational quadratic space of signature $(2, n)$ with $n\geq 3$ and Witt index $2$. 
Let $X={\D}/{\G}$ be an associated orthogonal modular variety, 
where ${\D}$ is an open set of the isotropic quadratic in ${\proj}V_{{\C}}$ and 
${\G}$ is a neat subgroup of ${\rm O}(V)$. 
We have $\dim X=n$. 
The ${\Q}$-rank is $2$, and we have only $0$-dimensional and $1$-dimensional cusps which 
correspond to isotropic lines and isotropic planes in $V$ respectively. 
If $F$ is a $0$-dimensional cusp and $I$ is the corresponding isotropic line, 
then $Y_{F}=X_{F}$, $U(F)_{{\Q}}$ is identified with the Lorentzian space $I\otimes_{{\Q}} (I^{\perp}/I)\simeq I^{\perp}/I$, 
and ${\GF}$ is a neat subgroup of ${\rm O}(I^{\perp}/I)$. 

Since $n(1)=1$ and $n(2)=n$, 
${\Gr}^{W}_{l}H^{2n-k}(X)$ has only Hodge classes if $l>2n-k+1$ (\cite{Ma2}). 
Thus we have practically only corank spectral sequence of level $p=0$ which computes ${\Gr}^{W}_{2n}H^{2n-\ast}(X)$. 
By Proposition \ref{prop: E1 degeneration}, this gives 
\begin{equation*}
{\Gr}^{W}_{2n}H^{2n-k}(X) \simeq \bigoplus_{F\in \mathcal{C}(2)} H^{n-k}({\GF}) 
\end{equation*}
for $3\leq k \leq n$. 
Cohomology of hyperbolic groups enters here. 
In the remaining case $k=2$, we have  
\begin{equation*}
{\Gr}^{W}_{2n}H^{2n-2}(X) \simeq 
\ker \left( \bigoplus_{F\in \mathcal{C}(2)}H^{n-2}({\GF}) \to \bigoplus_{F\in \mathcal{C}(1)} {\C} \right) 
\end{equation*}
by \eqref{eqn: E3/2}.

\appendix

\section{Retraction via barycentric subdivision}\label{sec: appendix}

Let $F$ be a cusp of a Hermitian symmetric domain ${\D}$ acted on by a neat arithmetic group ${\G}$ as in \S \ref{sec: modular}. 
We recall some notations from \S \ref{sec: modular}. 
Let $C(F)$ be the distinguished open homogeneous cone in $U(F)$ 
and $C(F)^{\ast}\subset U(F)$ be the union of $C(F)$ with the boundary cones $C(F')\subset U(F')$ for $F' \succ F$. 
Let $\Sigma_{F}$ be a ${\GF}$-admissible cone decomposition of $C(F)^{\ast}$, 
and $\partial\Sigma_{F}$ be its restriction to the boundary $C(F)^{\ast}-C(F)$. 
We assume that $\Sigma_{F}$ is \textit{simplicial}, i.e., 
every cone $\sigma\in \Sigma_{F}$ is generated by $\dim \sigma$ vectors, 
and satisfies the following two conditions: 
\begin{enumerate}
\item Two rays of a cone $\sigma\in \Sigma_{F}$ are not ${\GF}$-equivalent. 
\item For each cone $\sigma\in \Sigma_{F}$, the intersection $\sigma \cap \partial\Sigma_{F}$ 
is contained in a single boundary component $U(F')$ and hence is a face of $\sigma$. 
\end{enumerate} 
Both properties descend to ${\GF}$-invariant subdivisions. 
Our fan in \S \ref{sec: corank ss} satisfies (1) by the SNC condition. 
The property (2) can be realized by taking the barycentric subdivision before taking the smooth projective SNC subdivision.  

Let $(\Delta_{F}, \partial\Delta_{F})$ be the projectivization of $(\Sigma_{F}, \partial\Sigma_{F})$. 
Then $\Delta_{F}$ is an infinite simplicial complex, $\partial\Delta_{F}$ is a subcomplex of $\Delta_{F}$, 
and $\Delta_{F}$ is locally finite outside $\partial\Delta_{F}$. 
In what follows, we use the notation $\sigma$ exclusively for simplices; 
cones will no longer appear except in the beginning of \S \ref{ssec: glue}. 
We write $\Delta_{F}-\partial\Delta_{F}$ for the topological space underlying $\Delta_{F}$ minus that underlying $\partial\Delta_{F}$. 
We denote by $C_{c}^{\bullet}(\Delta_{F}, \partial\Delta_{F})$ the complex of relative simplicial cochains 
with compact support (i.e., nonzero for only finitely many simplices), 
which makes sense by the local finiteness of $\Delta_{F}$ outside $\partial\Delta_{F}$ (cf.~\cite{Hatcher} p.242). 
The purpose of this appendix is to prove the following proposition 
which was used in the proof of Proposition \ref{prop: acyclic}.

\begin{proposition}\label{prop: retraction}
There exists a simplicial complex $\Delta_{F}^{+}$, constructed as a subcomplex of the barycentric subdivision of $\Delta_{F}$, 
and a subcomplex $\partial\Delta_{F}^{+}$ of $\Delta_{F}^{+}$ 
satisfying the following properties: 

(1) $\Delta_{F}^{+}$ is locally finite and ${\GF}$-invariant. 

(2) $\Delta_{F}^{+}-\partial\Delta_{F}^{+}$ is homeomorphic to $\Delta_{F}-\partial\Delta_{F}$. 

(3) $C_{c}^{\bullet}(\Delta_{F}^{+}, \partial\Delta_{F}^{+})$ is quasi-isomorphic to $C_{c}^{\bullet}(\Delta_{F}, \partial\Delta_{F})$. 
\end{proposition}

The pair $(\Delta_{F}^{+}, \partial\Delta_{F}^{+})$ can be viewed as a retraction of $(\Delta_{F}, \partial\Delta_{F})$ toward interior. 
Its construction as well as the proof of the required properties are done in two steps. 
In \S \ref{ssec: single simplex}, we give a general construction of retraction 
for a single simplex in an Euclidean space. 
In \S \ref{ssec: glue}, we glue the construction in \S \ref{ssec: single simplex} 
to obtain the retraction $(\Delta_{F}^{+}, \partial\Delta_{F}^{+})$ globally. 
The canonicity of barycentric subdivision is crucial for making the gluing possible. 

Our construction is general enough so that 
it also applies to general finite simplicial complexes. 
With extra effort, it can also be extended to infinite simplicial complexes with certain group action (Remark \ref{remark: group action}), 
of which Proposition \ref{prop: retraction} is a special case. 
However, for simplicity of exposition, 
we take a short-cut in the gluing process by making use of the integral structure ${\UFZ}=U(F)\cap {\G}$ of $U(F)$, 
which is available only in the present specific situation.

\subsection{Single simplex}\label{ssec: single simplex}

Let $\Delta$ be a $d$-simplex in an Euclidean space ${\R}^N$, 
i.e., an (ordered) collection of $d+1$ points with $d\leq N$ in linearly general position. 
Elements of $\Delta$ are called the vertices of $\Delta$. 
As usual, we tacitly identify $\Delta$ with the convex hull of the vertices. 
Let $\Delta_{0}$ be a face of $\Delta$. 
We allow $\Delta_{0}$ to be empty or the whole $\Delta$. 
Let $\Delta_{1}\subset \Delta$ be the face consisting vertices not contained in $\Delta_{0}$. 

Let $\Delta_{brc}$ be the barycentric subdivision of $\Delta$ as defined in \cite{Hatcher} p.119. 
Simplices $\sigma$ of $\Delta_{brc}$ are in one-to-one correspondence with 
flags $\sigma_{0} \prec \cdots \prec \sigma_{k}$ of faces of $\Delta$. 
The vertices of $\sigma$ are the barycenters of $\sigma_{0}, \cdots, \sigma_{k}$. 

\begin{lemma}\label{lem: alternative}
A top-dimensional simplex $\sigma$ of $\Delta_{brc}$ intersects either with $\Delta_{0}$ or with $\Delta_{1}$. 
In the former case $\sigma$ is disjoint from $\Delta_{1}$; 
in the latter case $\sigma$ is disjoint from $\Delta_{0}$. 
\end{lemma}

\begin{proof}
Since $\sigma$ is top-dimensional, it corresponds to a maximal flag 
$\sigma_{0} \prec \cdots \prec \sigma_{d}$. 
Then this alternative is determined according to which $\Delta_{0}$ or $\Delta_{1}$ 
the first member $\sigma_{0}$ (vertex) of the flag belongs to. 
\end{proof}

\begin{definition}\label{def: retraction}
We define a subcomplex $\Delta^{+}$ of $\Delta_{brc}$ 
as the union of all top-dimensional simplices meeting $\Delta_{1}$ ($\Leftrightarrow$ disjoint from $\Delta_{0}$). 
We also define a subcomplex $\partial\Delta^{+}$ of $\Delta_{brc}$ as the union of all simplices 
which are disjoint from both $\Delta_{0}$ and $\Delta_{1}$. 
We call the pair $(\Delta^{+}, \partial\Delta^{+})$ the \textit{retraction} of $(\Delta, \Delta_{0})$. 
\end{definition}

See Figure \ref{figure: retraction} for an illustration of $(\Delta^{+}, \partial\Delta^{+})$ in the case $d=2$. 
In the extreme case $\Delta_{0}=\emptyset$, we have $(\Delta^{+}, \partial\Delta^{+}) = (\Delta_{brc}, \emptyset)$; 
in the opposite case $\Delta_{0}=\Delta$, we have $\Delta^{+}=\emptyset$. 
We write $\Delta^{-}$ for the union of all top-dimensional simplices of $\Delta_{brc}$ meeting $\Delta_{0}$. 
Then $\Delta_{brc} = \Delta^{+}\cup \Delta^{-}$ by Lemma \ref{lem: alternative}.

\begin{figure}[h]
\includegraphics[height=45mm, width=60mm]{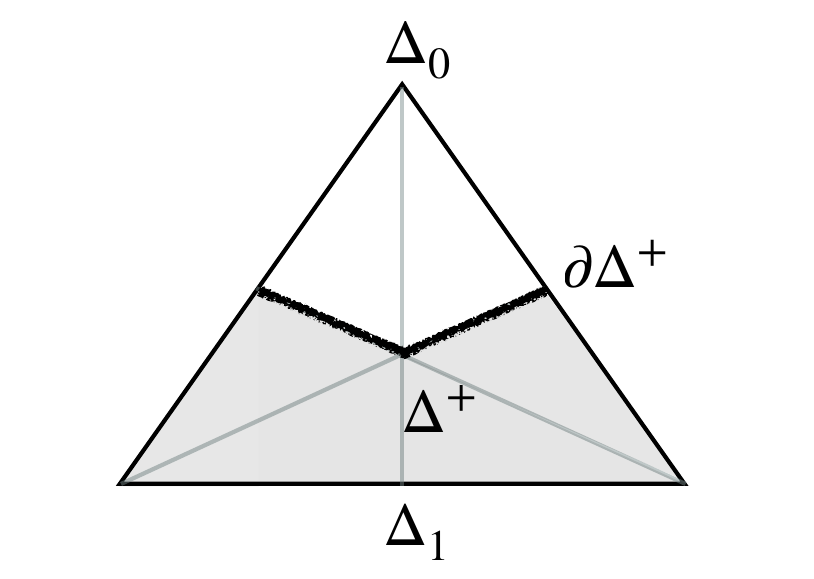}
\caption{$(\Delta^{+}, \partial\Delta^{+})$\label{figure: retraction}}
\end{figure}

\begin{lemma}\label{lem: Sigma+ criterion}
Let $\sigma$ be a simplex of $\Delta_{brc}$ and 
$\sigma_{0} \prec \cdots \prec \sigma_{k}$ be the corresponding flag of faces of $\Delta$. 

(1) We have 
\begin{equation*}
\sigma \subset \Delta^{+} \; \; \Leftrightarrow \; \; 
\sigma \cap \Delta_{0} = \emptyset \; \; \Leftrightarrow \; \; 
\sigma_{0} \cap \Delta_{1} \ne \emptyset \; \; \Leftrightarrow \; \; 
\sigma_{0} \not\subset \Delta_{0}. 
\end{equation*}

(2) $\sigma$ is contained in $\partial\Delta^{+}$ if and only if 
$\sigma_{0}\cap \Delta_{0} \ne \emptyset$ and $\sigma_{0}\cap \Delta_{1} \ne \emptyset$. 
In particular, $\partial \Delta^{+}$ is a subcomplex of $\Delta^{+}$ 
and we have $\partial\Delta^{+}=\Delta^{+}\cap \Delta^{-}$. 
\end{lemma}

\begin{proof}
(1) We have $\sigma\subset \Delta^{+}$ if and only if 
$\sigma_{0} \prec \cdots \prec \sigma_{k}$ can be extended to a maximal flag 
$\sigma_{0}' \prec \cdots \prec \sigma_{d}'$ such that the vertex $\sigma_{0}'$ belongs to $\Delta_{1}$ 
(see the proof of Lemma \ref{lem: alternative}). 
This is equivalent to $\sigma_{0}\cap \Delta_{1} \ne \emptyset$. 
The equivalence 
$\sigma_{0} \cap \Delta_{1} \ne \emptyset \Leftrightarrow \sigma_{0} \not\subset \Delta_{0}$ 
is obvious. 
If $\sigma\cap \Delta_{0} = \emptyset$, 
the barycenter of $\sigma_{0}$ is not contained in $\Delta_{0}$, and so $\sigma_{0} \not\subset \Delta_{0}$. 
Conversely, if $\sigma\cap \Delta_{0} \ne \emptyset$, 
one of the vertices of $\sigma$ is contained in $\Delta_{0}$, 
which implies that one of the members of the flag $\sigma_{0} \prec \cdots \prec \sigma_{k}$ 
is contained in $\Delta_{0}$. 
In particular, we have $\sigma_{0}\subset \Delta_{0}$. 

(2) The first assertion follows from the equivalence of the second and the third conditions of (1) 
for both $\Delta^{+}$ and $\Delta^{-}$. 
The equality $\partial\Delta^{+}=\Delta^{+}\cap \Delta^{-}$ 
follows from the equivalence of the first and the second conditions of (1) 
for both $\Delta^{+}$ and $\Delta^{-}$.  
\end{proof}

\begin{corollary}\label{cor: partialDelta+}
When $\Delta_{0}, \Delta_{1}\ne \emptyset$, $\partial \Delta^{+}$ is of codimension $1$ in $\Delta^{+}$. 
\end{corollary}

\begin{proof}
Every flag $\sigma_{0}\prec \cdots \prec \sigma_{k}$ with 
$\sigma_{0}\cap \Delta_{0} \ne \emptyset$ and $\sigma_{0}\cap \Delta_{1} \ne \emptyset$ 
can be extended to a flag 
$\sigma_{0}'\prec \cdots \prec \sigma_{d-1}'$ of length $d$ with 
$\sigma_{0}'\cap \Delta_{0} \ne \emptyset$ and $\sigma_{0}'\cap \Delta_{1} \ne \emptyset$, 
by taking $\sigma_{0}'$ to be a $1$-dimensional face of $\sigma_{0}$ 
one of whose vertex is from $\sigma_{0}\cap \Delta_{0}$ and 
another from $\sigma_{0}\cap \Delta_{1}$. 
\end{proof}

The construction of $(\Delta^{+}, \partial\Delta^{+})$ is compatible with restriction as follows. 
Let $\sigma$ be a face of $\Delta$. 
If we put $\sigma_{0}=\sigma \cap \Delta_{0}$, 
we can construct the retraction of $(\sigma, \sigma_{0})$ 
by the same procedure as for $(\Delta, \Delta_{0})$. 
We denote it by $(\sigma^{+}, \partial\sigma^{+})$. 

\begin{lemma}\label{lem: retract restriction}
We have $\sigma^{+}=\sigma\cap \Delta^{+}$ and 
$\partial\sigma^{+}=\sigma\cap \partial\Delta^{+}$. 
\end{lemma}

\begin{proof}
The barycentric subdivision $\sigma_{brc}$ of $\sigma$ is the restriction of $\Delta_{brc}$ to $\sigma$. 
Then, by comparing Lemma \ref{lem: Sigma+ criterion} (1) for both $(\Delta, \Delta_{0})$ and $(\sigma, \sigma_{0})$, 
we see that a simplex of $\sigma_{brc}$ is contained in $\sigma^{+}$ if and only if it is contained in $\Delta^{+}$. 
The same holds for $\partial\sigma^{+}$. 
\end{proof}

\begin{corollary}\label{cor: boundary retraction}
The boundary of $\Delta^{+}$ (in the usual sense) is the union 
\begin{equation*}
\partial\Delta^{+} \cup \bigcup_{{\rm codim}(\sigma)=1}\sigma^{+}, 
\end{equation*}
where $\sigma$ ranges over codimension $1$ faces of $\Delta$. 
The part $\bigcup_{\sigma}\sigma^{+}$ is the intersection of $\Delta^{+}$ with the boundary of $\Delta$, 
and $\partial\Delta^{+}$ is the closure of the intersection of  the boundary of $\Delta^{+}$ with the interior of $\Delta$. 
\end{corollary}

\begin{proof}
The assertion that $\bigcup_{\sigma}\sigma^{+}$ is the intersection of $\Delta^{+}$ with the boundary of $\Delta$ 
follows from Lemma \ref{lem: retract restriction} for codimension $1$ faces of $\Delta$. 
Since $\Delta=\Delta^{+}\cup \Delta^{-}$, 
the remaining part of the boundary of $\Delta^{+}$ is $\Delta^{+}\cap \Delta^{-} = \partial\Delta^{+}$. 
\end{proof}

Next we construct a homeomorphism between the interiors. 
This makes essential use of the realization of $\Delta$ inside ${\R}^{N}$. 

\begin{lemma}\label{lem: retraction map}
We have a homeomorphism 
$\Delta - \Delta_{0} \stackrel{\simeq}{\to} \Delta^{+}-\partial\Delta^{+}$ 
which is identity on $\Delta_{1}$ 
and is compatible with restriction to the faces of $\Delta$. 
\end{lemma} 

\begin{proof}
We may assume that $\Delta_{0}, \Delta_{1} \ne \emptyset$. 
For $x\in \Delta_{0}$ and $y\in \Delta_{1}$, 
let $I_{x,y}$ be the line segment joining $x$ and $y$. 
Points on $I_{x,y}$ are expressed as $tx+(1-t)y$ with $0\leq t \leq 1$. 
Every point of $\Delta-\Delta_{0}-\Delta_{1}$ lies on a unique such line segment. 
Indeed, $\Delta$ is the union of convex hulls of $x$ and $\Delta_{1}$ over all $x\in \Delta_{0}$, 
and two such convex hulls intersect only at $\Delta_{1}$. 

\begin{claim}\label{claim: Ixy}
$I_{x,y}$ intersects with $\partial\Delta^{+}$ only once. 
\end{claim}

\begin{proof}
We normalize $\Delta$ to be the standard $d$-simplex in ${\R}^{d+1}$. 
We denote by $p_{0}, \cdots, p_{d}$ the vertices of $\Delta$. 
Let $\sigma$ be a top-dimensional simplex of $\Delta_{brc}$ and 
$p_{i_{0}}\prec p_{i_{0}}p_{i_{1}} \prec \cdots \prec p_{i_{0}}\cdots p_{i_{d}}$ be the corresponding flag of $\Delta$. 
A point $p$ of $\Delta$ belongs to $\sigma$ if and only if 
the order $p_{i_{0}}, \cdots, p_{i_{d}}$ agrees with that of distance from $p$. 
We express this situation as $p_{i_{0}}< \cdots <p_{i_{d}}$. 
(We allow $\leq$, but omit this for simplicity.) 
Thus a point $p$ determines an order of the vertices by the distance from $p$, 
which in turn determines the subdivided simplex containing $p$. 

Now we let $p$ move on $I_{x,y}$ from $x\in \Delta_{0}$ to $y\in \Delta_{1}$. 
At the starting point $p=x$, 
every vertex of $\Delta_{0}$ is nearer to $p$ than those of $\Delta_{1}$. 
Thus, if we denote by $p_{0,\ast}$, $p_{1,\ast}$ the vertices of $\Delta_{0}$, $\Delta_{1}$ respectively, 
the order around $p=x$ can be written as 
\begin{equation*}
p_{0,1} < p_{0,2} < \cdots < p_{0,m} < p_{1,1} < \cdots < p_{1,l}. 
\end{equation*}
As $p$ moves on $I_{x,y}$ toward $y\in \Delta_{1}$, this order changes: 
the vertices $p_{1,\ast}$ of $\Delta_{1}$ gradually shift to left 
without changing the order inside $\Delta_{1}$. 
(Actually we do not need the last property.) 
The point here is that once $p_{1,i}$ overtakes $p_{0,j}$, $p_{0,j}$ does not overtake $p_{1,i}$ after that:  
no further turnover among them. 
This is because $p$ moves on a line. 
When $p$ arrives around $y$, the order becomes 
 \begin{equation*}
p_{1,1} < \cdots < p_{1,l} < p_{0,1} < \cdots < p_{0,m}. 
\end{equation*}
As long as $p_{0,1}$ stays at the left end, $p$ belongs to $\Delta^{-}$; 
after $p_{1,1}$ comes to the left end, $p$ belongs to $\Delta^{+}$ (cf.~the proof of Lemma \ref{lem: alternative}). 
Thus the moment $p_{1,1}$ overtakes $p_{0,1}$ is the intersection point of $I_{x,y}$ with $\partial \Delta^{+}$, 
and this occurs only once. 
\end{proof}

We go back to the proof of Lemma \ref{lem: retraction map}. 
By Claim \ref{claim: Ixy}, there is a value $0<t(x,y)<1$ such that 
$tx+(1-t)y$ is contained in $\Delta^{+}$ if $t\leq t(x, y)$, 
and contained in $\Delta^{-}$ if $t\geq t(x, y)$. 
As a function of $(x, y)$, $t(x, y)$ is piecewise linear. 
Now the multiplication by $t(x, y)$
\begin{equation*}
[0, 1] \to [0, \: t(x, y)], \quad t\mapsto t(x, y)\cdot t, 
\end{equation*}
gives a bijection 
$I_{x,y}\to I_{x,y}\cap \Delta^{+}$ which is identity on $I_{x,y}\cap \Delta_{1}= \{ y \}$. 
By varying $(x, y)$, we obtain a homeomorphism 
$\Delta-\Delta_{0} \to \Delta^{+}-\partial\Delta^{+}$ 
which is identity on $\Delta_{1}$. 
Since this is defined on each $I_{x,y}$, 
compatibility with restriction follows from Lemma \ref{lem: retract restriction}. 
\end{proof}

Finally, we compare cochain complexes. 
We define a retraction map 
\begin{equation}\label{eqn: qis single}
R : C^{\bullet}(\Delta^{+}, \partial\Delta^{+}) \to C^{\bullet}(\Delta, \Delta_{0}) 
\end{equation}
between the relative simplicial cochain complexes as follows. 
Let $\varphi\in C^{k}(\Delta^{+}, \partial\Delta^{+})$. 
For a $k$-simplex $\sigma$ of $\Delta$, we write the barycentric subdivision of $\sigma$ as 
$\sigma=\sum_{i}\sigma^{-}_{i}+\sum_{j}\sigma^{+}_{j}$ 
where $\sigma^{-}_{i}$ intersects with $\Delta_{0}$ and $\sigma^{+}_{j}$ intersects with $\Delta_{1}$. 
We write $\sigma^{+}=\sum_{j}\sigma^{+}_{j}$. 
(This notation is compatible with that in Lemma \ref{lem: retract restriction}.) 
Then we define $R\varphi \in C^{k}(\Delta)$ by 
\begin{equation*}
(R\varphi) (\sigma) = \varphi(\sigma^{+}). 
\end{equation*}
If $\sigma\subset \Delta_{0}$, we have $\sigma^{+}=0$ and so $(R\varphi)(\sigma)=0$. 
Hence $R\varphi \in C^{k}(\Delta, \Delta_{0})$. 

\begin{lemma}\label{lem: cochain map single}
$R$ is a cochain map and is a quasi-isomorphism. 
\end{lemma}

\begin{proof}
Let $\varphi\in C^{k}(\Delta^{+}, \partial\Delta^{+})$. 
We take a $(k+1)$-dimensional face $\sigma$ of $\Delta$ not contained in $\Delta_{0}$. 
We denote by $\delta$ and $\partial'$ the coboundary and boundary maps respectively. 
(The notation $\partial$ was already used in Definition \ref{def: retraction}.) 
Then, by Corollary \ref{cor: boundary retraction} for $(\sigma, \sigma_{0})$, we have 
\begin{equation*}
(R\circ \delta(\varphi))(\sigma) - (\delta\circ R(\varphi))(\sigma) 
\: = \: 
\varphi( \partial'(\sigma^{+}) - (\partial'\sigma)^{+})  
\: = \: 
\varphi(\partial\sigma^{+}).  
\end{equation*}
Since $\partial\sigma^{+}$ is contained in $\partial\Delta^{+}$, it is annihilated by $\varphi$. 
Hence $R\circ \delta = \delta\circ R$. 

It remains to verify that $R$ is a quasi-isomorphism. 
When $\Delta_{0}\ne \emptyset$, this is obvious because both 
$H^{\bullet}(\Delta^{+}, \partial\Delta^{+})$ and $H^{\bullet}(\Delta, \Delta_{0})$ vanish. 
When $\Delta_{0}=\emptyset$, 
we have $(\Delta^{+}, \partial\Delta^{+})=(\Delta_{brc}, \emptyset)$. 
Then $R$ is the ordinary subdivision map for cochains and hence is a quasi-isomorphism. 
\end{proof}

\subsection{Gluing}\label{ssec: glue}

We go back to the simplicial pair $(\Delta_{F}, \partial\Delta_{F})$ considered in the beginning of Appendix. 
The simplicial complex $\Delta_{F}$ was defined as the projectivization of the cone complex $\Sigma_{F}$, 
but in fact it can be realized inside $U(F)$ (though the topology at $\partial\Delta_{F}$ is different) as follows. 
Let $\sigma$ be a $k$-simplex of $\Delta_{F}$ and $\tilde{\sigma}\in \Sigma_{F}$ be the corresponding cone. 
By using the ${\Z}$-structure ${\UFZ}$ of $U(F)$, 
the rays of $\tilde{\sigma}$ can be uniquely written as 
${\R}_{\geq 0}v_{0}, \cdots, {\R}_{\geq 0}v_{k}$ 
with $v_{i}$ a primitive vector in ${\UFZ}$. 
Then we identify $\sigma$ with the convex hull of $v_{0}, \cdots, v_{k}$ inside $U(F)$. 
This embedding $\sigma\hookrightarrow U(F)$ is compatible with restriction to the faces of $\sigma$. 
Therefore we can glue them over all $\sigma\in \Delta_{F}$ to obtain an embedding 
\begin{equation*}
\Delta_{F}\hookrightarrow U(F). 
\end{equation*}
By the local finiteness outside $\partial\Delta_{F}$, 
the topology on $\Delta_{F} - \partial\Delta_{F}$ agrees with that induced from $U(F)$. 
Since ${\GF}$ preserves ${\UFZ}$, this embedding is ${\GF}$-equivariant. 
In what follows, we regard $\Delta_{F}$ as a subset of $U(F)$ in this way. 

For each simplex $\sigma$ of $\Delta_{F}$, we define its face $\sigma_{0}$ by 
$\sigma_{0}=\sigma\cap \partial\Delta_{F}$. 
(Recall that this is a face of $\sigma$ by our assumption (2) in the beginning of Appendix.) 
Since $\sigma$ is now realized as a simplex in the Euclidean space $U(F)$, 
we can take the retraction $(\sigma^{+}, \partial\sigma^{+})$ of $(\sigma, \sigma_{0})$ 
by the procedure in Definition \ref{def: retraction}. 
By Lemma \ref{lem: retract restriction}, this construction can be naturally glued with that for adjacent simplices. 
In this way we obtain the simplicial pair 
\begin{equation}\label{eqn: gluing}
(\Delta_{F}^{+}, \partial\Delta_{F}^{+}) = \bigcup_{\sigma\in \Delta_{F}}(\sigma^{+}, \partial\sigma^{+}). 
\end{equation}
We shall prove that $(\Delta_{F}^{+}, \partial\Delta_{F}^{+})$ satisfies the required properties. 
 
\begin{lemma}
$\Delta_{F}^{+}$ is locally finite and ${\GF}$-invariant. 
\end{lemma}

\begin{proof}
Since $\Delta_{F}$ is ${\GF}$-invariant, 
the canonicity of the construction $\sigma \rightsquigarrow \sigma^{+}$ inside $U(F)$ 
implies that $\Delta_{F}^{+}$ is also ${\GF}$-invariant. 
Since the simplices of $\Delta_{F}^{+}$ are now disjoint from $\partial\Delta_{F}$, 
the local finiteness of  $\Delta_{F}^{+}$ follows from that of $\Delta_{F}$ outside $\partial\Delta_{F}$. 
\end{proof}

By gluing the retraction homeomorphisms in Lemma \ref{lem: retraction map}, 
we obtain a homeomorphism 
\begin{equation*}
\Delta_{F}-\partial\Delta_{F} \to \Delta_{F}^{+}-\partial\Delta_{F}^{+}. 
\end{equation*}
It remains to compare the cochain complexes. 
We define a map 
\begin{equation*}
R : C_{c}^{\bullet}(\Delta_{F}^{+}, \partial\Delta_{F}^{+}) \to C_{c}^{\bullet}(\Delta_{F}, \partial\Delta_{F}) 
\end{equation*}
by the same procedure as \eqref{eqn: qis single}. 
Clearly this sends cochains with compact support to cochains with compact support. 
The proof of the first part of Lemma \ref{lem: cochain map single} is still valid 
and shows that $R$ is a cochain map. 

\begin{lemma}\label{lem: qis full}
$R$ is a quasi-isomorphism. 
\end{lemma}

\begin{proof}
We write $G={\GF}$. 
We proceed inductively by using Mayer-Vietoris argument for $G$-orbits of simplices. 

We choose a top-dimensional simplex $\sigma\subset \Delta_{F}$. 
Let $\Delta_{1}\subset \Delta_{F}$ be the union of simplices $G$-equivalent to $\sigma$, 
and $\Delta_{2}\subset \Delta_{F}$ be the union of the remaining top-dimensional simplices. 
Then $\Delta_{1}, \Delta_{2}$ and $\Delta_{12}=\Delta_{1} \cap \Delta_{2}$ 
are $G$-invariant subcomplexes of $\Delta_{F}$. 
We put $\partial\Delta_{\ast}=\Delta_{\ast}\cap \partial\Delta_{F}$ 
for $\ast=1, 2, 12$. 
We can construct the retraction 
$(\Delta_{\ast}^{+}, \partial\Delta_{\ast}^{+})$ of $(\Delta_{\ast}, \partial\Delta_{\ast})$ 
in the same way as \eqref{eqn: gluing}. 
By construction we have 
\begin{equation}\label{eqn: union DeltaF+}
(\Delta_{F}^{+}, \partial\Delta_{F}^{+}) = (\Delta_{1}^{+}, \partial\Delta_{1}^{+}) \cup (\Delta_{2}^{+}, \partial\Delta_{2}^{+}), 
\end{equation}
\begin{equation*}
(\Delta_{12}^{+}, \partial\Delta_{12}^{+}) = (\Delta_{1}^{+}, \partial\Delta_{1}^{+}) \cap (\Delta_{2}^{+}, \partial\Delta_{2}^{+}). 
\end{equation*}
The cochain map $R$ is defined for each pair $(\Delta_{\ast}, \partial\Delta_{\ast})$. 
Then we have the commutative diagram 
\begin{equation}\label{eqn: MV comparison}
\xymatrix@C=15pt{
0 \ar[r] & C^{\bullet}_{c}(\Delta_{F}^{+}, \partial\Delta_{F}^{+}) \ar[d]^{R} \ar[r] & 
C^{\bullet}_{c}(\Delta_{1}^{+}, \partial\Delta_{1}^{+}) \oplus C^{\bullet}_{c}(\Delta_{2}^{+}, \partial\Delta_{2}^{+}) 
\ar[d]^{R} \ar[r] & C^{\bullet}_{c}(\Delta_{12}^{+}, \partial\Delta_{12}^{+}) \ar[d]^{R} \ar[r] & 0 \\ 
0 \ar[r] & C^{\bullet}_{c}(\Delta_{F}, \partial\Delta_{F}) \ar[r] & 
C^{\bullet}_{c}(\Delta_{1}, \partial\Delta_{1}) \oplus C^{\bullet}_{c}(\Delta_{2}, \partial\Delta_{2}) 
 \ar[r] & C^{\bullet}_{c}(\Delta_{12}, \partial\Delta_{12})  \ar[r] & 0 \\ 
}
\end{equation}
where the upper exact sequence is the Mayer-Vietoris for \eqref{eqn: union DeltaF+} 
and the lower one is the Mayer-Vietoris for 
$(\Delta_{F}, \partial\Delta_{F}) = (\Delta_{1}, \partial\Delta_{1}) \cup (\Delta_{2}, \partial\Delta_{2})$. 
What is required for running induction with this diagram is the following. 

\begin{claim}\label{claim: MV start}
Let $\sigma$ be a simplex of $\Delta_{F}$ (not necessarily top-dimensional) 
and $\Delta'\subset \Delta_{F}$ be the subcomplex obtained as the union of simplices $G$-equivalent to $\sigma$. 
Let $\partial\Delta'=\Delta'\cap \partial\Delta_{F}$. 
Then the retracting cochain map  
\begin{equation}\label{eqn: R G-orbit}
R : C_{c}^{\bullet}((\Delta')^{+}, \partial(\Delta')^{+}) \to C_{c}^{\bullet}(\Delta', \partial\Delta') 
\end{equation}
is a quasi-isomorphism. 
\end{claim}

\begin{proof}
When $\sigma\subset \partial\Delta_{F}$, 
we have $\Delta'=\partial\Delta'$ and $(\Delta')^{+}=\emptyset$. 
Hence both 
$C_{c}^{\bullet}(\Delta', \partial\Delta')$ and  
$C_{c}^{\bullet}((\Delta')^{+}, \partial(\Delta')^{+})$ are trivial. 
In what follows, we assume $\sigma \not\subset \partial\Delta_{F}$. 

If $\sigma \cap \gamma\sigma$ is non-empty for $\gamma \ne {\rm id}\in G$, 
it is contained in $\partial\Delta_{F}$ because the vertices of $\sigma \cap \gamma\sigma$ 
are fixed by $\gamma$ by our assumption (1) in the beginning of Appendix 
and $G$ acts on $\Delta_{F}$ freely outside $\partial\Delta_{F}$. 
Therefore, if $\tau$ is a simplex of $\Delta'$ not contained in $\partial\Delta'$, 
it is a face of a unique translate $\gamma\sigma$ of $\sigma$. 
This shows that we have the decomposition 
\begin{equation*}
C_{c}^{\bullet}(\Delta', \partial\Delta') = 
\bigoplus_{\gamma\in G} C_{c}^{\bullet}(\gamma\sigma, \partial(\gamma\sigma)) 
\end{equation*}
where 
$\partial(\gamma\sigma) = \gamma\sigma \cap \partial\Delta_{F}$. 
On the other hand, since $\sigma^{+}$ is now disjoint from $\partial \Delta_{F}$, 
we have $\sigma^{+} \cap \gamma\sigma^{+} = \emptyset$ if $\gamma\ne {\rm id}$. 
Therefore 
$(\Delta')^{+}=\sqcup_{\gamma}\gamma\sigma^{+}$ 
and 
$\partial(\Delta')^{+}=\sqcup_{\gamma} \partial(\gamma\sigma^{+})$. 
This implies     
\begin{equation*}
C_{c}^{\bullet}((\Delta')^{+}, \partial(\Delta')^{+}) = 
\bigoplus_{\gamma\in G} C_{c}^{\bullet}(\gamma\sigma^{+}, \partial(\gamma\sigma^{+})).  
\end{equation*}
It follows that \eqref{eqn: R G-orbit} is the direct sum of the retraction maps 
for each $(\gamma\sigma, \partial(\gamma\sigma))$, 
which is a quasi-isomorphism by Lemma \ref{lem: cochain map single}. 
\end{proof}

Now, by Claim \ref{claim: MV start} and induction on the number of $G$-orbits of top-dimensional simplices, 
the middle $R$ in \eqref{eqn: MV comparison} is a quasi-isomorphism. 
By induction on the dimension (using Claim \ref{claim: MV start} similarly), 
the right $R$ is also a quasi-isomorphism. 
By the five lemma, we conclude that the left $R$ in \eqref{eqn: MV comparison} is a quasi-isomorphism. 
\end{proof}

The proof of Proposition \ref{prop: retraction} is now completed.

\begin{remark}\label{remark: group action}
Proposition \ref{prop: retraction} holds more generally for a simplicial pair 
$(\Delta, \partial\Delta)$ acted on by a discrete group $G$ satisfying the following conditions: 
\begin{enumerate}
\item $\Delta$ is locally finite outside $\partial\Delta$. 
\item For any simplex $\sigma$ of $\Delta$, $\sigma\cap \partial\Delta$ is a face of $\sigma$. 
\item Two vertices of a simplex of $\Delta$ are not $G$-equivalent. 
\item The stabilizer of a vertex not contained in $\partial\Delta$ is trivial. 
\item There are only finitely many simplices up to $G$-equivalence. 
\end{enumerate}
Indeed, many arguments in \S \ref{ssec: glue} are valid in this generality 
except the realization inside $U(F)$. 
To avoid this short-cut, we take the quotient $\Delta'=\Delta/G$, which is a finite $\Delta$-complex, 
then take the simplicial complex $\Delta''$ which have the same vertices as $\Delta'$, 
and finally realize $\Delta''$ in an Euclidean space ${\R}^N$. 
This embeds each simplex of $\Delta$ in ${\R}^{N}$ in a $G$-invariant way. 
This enables us to take a $G$-invariant retraction. 
\end{remark}


\end{document}